\documentclass[11pt]{amsart}
\usepackage{fullpage}
\usepackage{amssymb, amsmath, wasysym, mathrsfs, mathtools, color, stmaryrd, enumerate}
\usepackage[all, cmtip]{xy}
\usepackage{url}
\usepackage{tikz-cd}
\usetikzlibrary{calc} 

\definecolor{hot}{RGB}{65,105,225}
\usepackage[pagebackref=true,colorlinks=true, linkcolor=hot,  citecolor=hot, urlcolor=hot]{hyperref}

\newcommand{\commentOut}[1]{}

\theoremstyle{plain}
\newtheorem{theorem}{Theorem}[section]
\newtheorem{prop}[theorem]{Proposition}

\newtheorem{cor}[theorem]{Corollary}

\newtheorem{lemma}[theorem]{Lemma}
\newtheorem{thrm}[theorem]{Theorem}
\newtheorem{claim}[theorem]{Claim}

\theoremstyle{definition}

\newtheorem{defn}[theorem]{Definition}
\newtheorem{rmk}[theorem]{Remark}

\newtheorem{ex}[theorem]{Example}
\newtheorem*{ex*}{Example}
\newtheorem*{problem}{Problem}
\newtheorem*{question}{Question}

\newtheorem*{acknowledgement}{Acknowledgement}

\newcommand{\cO}{{\mathcal O}}

\newcommand{\cF}{{\mathcal F}}

\newcommand{\cX}{\mathcal{X}}
\newcommand{\cY}{\mathcal{Y}}

\newcommand{\wK}{\widetilde{K}}

\newcommand{\bC}{{\mathbb{C}}}

\newcommand{\ubul}{{\,\begin{picture}(-1,1)(-1,-3)\circle*{2}\end{picture}\ }}

\DeclareMathOperator{\spec}{Spec}

\DeclareMathOperator{\Hom}{Hom}

\DeclareMathOperator{\Ho}{\mathrm{Hol}}
\DeclareMathOperator{\Au}{\mathrm{Aut}}

\newcommand\be{\begin{equation}}
\newcommand\ee{\end{equation}}

\title[Equivariant Grauert's approximation theorem and applications]{Equivariant deformations of isolated singularities and applications}

\author{An-Khuong Doan}\address{Institute for Advanced Studies in Mathematics, Harbin Institute of Technology, Harbin, China 150001 }\email{an-khuong.doan@hit.edu.cn}

\begin{document} 
\maketitle
\begin{abstract} 
In this paper, we develop an equivariant version of the classical theory of deformations of germs of complex spaces  in the presence of complex Lie groups. As a tradition, the main result is the existence of equivariant semi-universal deformations of germs of complex spaces in the reductive case and the non-existence in general in the non-reductive case. The latter is justified by an explicit counterexample. In particular, it generalizes Grauert-Donin's existence theorem to the equivariant settings and Ferrer-Puerta-Slodowy's one to the case of reductive complex Lie groups. Several applications to deformations of pairs, rigidity and actions on Milnor fibers are given.
\end{abstract}
\tableofcontents
\section{Introduction} Deformation theory was initiated by the revolutionary work of Kodaira-Spencer-Kuranishi, see \cite{Kod05, KS58, Kur65, Kur71} whose fundamental result is the existence of semi-universal deformations of a given compact complex manifold $X$ containing all information about its small deformations. Furthermore, when the automorphism group of $X$ is reductive, its natural action on $X$ extends to the semi-universal deformations to which we often refer as equivariant semi-universal deformations or equivariant Kuranishi families (cf. \cite{Doa21}). It allows to deduce several applications to the local structure of the moduli stack of integrable complex structures (cf. \cite{Doa24}) and the Techm\"{u}ller stack of compact quotients of $\mathrm{SL}_2(\mathbb{C})$ (cf. \cite{Jam24}); to analytic $K$-semistability and local wall-crossing (cf. \cite{ST25, Tip25}); to deformations of primitive Enriques varieties (cf. \cite{DTX25}); to algebraic approximations of compact K\"{a}hler threefolds (cf. \cite{Lin24}). 

In parallel, deformations of germs of complex spaces were also studied by many people, where the existence of semi-universal deformations were obtained for particular cases: complete intersection germs of complex spaces with isolated singularity \cite{KS72},  germs of complex spaces with isolated singularity and vanishing obstruction space \cite{Tju70}. Shortly after, Grauert definitively settled the existence problem in his seminal paper \cite{Gra72} for a large class of germs of complex spaces including the two former cases. The main tool is his extremely powerful approximation theorem - Grauert's approximation (cf. Theorem \ref{grauapp} below). Around the same year, I. Donin  made use of the theory of Banach spaces developed by A. Douady (cf. \cite{Dou64}) to derive the same result \cite{Don72}. For the reader's convenience, we shall briefly  recall these two approaches in the next section. We also refer the reader to Palamodov's surveys \cite{Pal76, Pal90} for more historical facts.  Inspired by \cite{Doa21, Doa22, Doa24}, it is natural to pose the following problem.

 \begin{problem} \label{mainquest} If $(X,0)$ is a germ of complex spaces with isolated singularities equipped further with an action of a (real) compact Lie group $G$ (or a reductive complex Lie group), does there exist a $G$-equivariant semi-universal deformation of $X_0$?
\end{problem} Let us now review the existing literature on this problem which in fact has a long history. When the singularity $X$ is algebraic, i.e. defined by polynomials, this problem was first attacked more than fifty years ago by  Pinkham who gave an positive answer for $G=\mathbb{C}^*$ in \cite{Pin74, Pin78}. Around the same time, Slodowy investigated systematically equivariant deformations of singularities where he gave partial responses, i.e. the existence of $G$-equivariant semi-universal deformations for $X$ an arbitrary affine variety with isolated singularity and $G$ a finite group in \cite{Slo78}; for $X$ a complete intersection  and $G$ a linearly reductive group (cf. \cite[$\S$2.5, Theorem and Remark 1]{Slo80}). The latter gives several contributions to the algebraic classification of singularities therein. In the same year, Rim showed that the answer is yes once again for  $X$ a general isolated algebraic singularity and $G$ a linearly reductive group. However, the result is unfortunately only formal, i.e. the semi-universal deformation and the actions constructed exist formally, leaving the convergence as an open problem, analytically (\cite{Rim80}, \cite[Section 3]{Slo01}).

 On the holomorphic side where singularities are defined by holomorphic functions, the problem was taken into consideration by Ferrer-Puerta in \cite{FP81, Pue82} and revisited by Hauser in \cite{Hau85} where the authors all showed that the answer is yes again when $G=\mathbb{C}^*$. The same thing was indicated in \cite[Theorem 2.5]{Gre20} mentioning vaguely a $\mathbb{C}^*$-equivariant version of Verdier-Grauert's construction of semi-universal deformations of isolated singularities given in \cite{GH74}. Since then, no further progress towards an answer to this problem has been made. Last but not least, it is worth noting that the confirmation of this problem has been existing as a folklore among experts, e.g. \cite[Paragraph before Example 3.10]{OSS16} and \cite[Section 3]{Slo01}.
 
 Therefore, the main aim of this paper is to give an  answer to the problem under consideration. Namely, we are going to prove that the answer is affirmative in the case that $G$ is reductive and is negative in general in the case that $G$ is not reductive. This effectively generalizes the main results of \cite{Gra72, Don72, FP81, Pue82, Hau85, Gre20} to the equivariant settings. A summary of approaches, precise results, applications and open problems shall be given in the next section.


\section{Possible approaches, main results, applications and open problems}

\subsection{Theory of Banach analytic spaces and resolutions of structural sheaves} Let $\mathcal{O}_{X,0}$ be the structural sheaf of a germ of complex spaces $(X,0)$.   Consider a free resolution $\mathcal{F}^\ubul$ of finite length of $\mathcal{O}_{X}$
\begin{equation*}\mathcal{F}^\ubul: 0 \longrightarrow \mathcal{F}^r \overset{\varphi_r}{\longrightarrow}\mathcal{F}^{r-1} \overset{\varphi_{r-1}}{\longrightarrow}\cdots \overset{\varphi_{1}}{\longrightarrow} \mathcal{F}^0  \overset{\pi}{\longrightarrow} \mathcal{O}_{X,0}\longrightarrow 0
\end{equation*}
 I. Donin proved that a deformation of $(X,0)$ is equivalent to a deformation of $\mathcal{F}^\ubul$ (cf. \cite[Proposition 1.3]{Don72}). Therefore, instead of constructing a semi-universal deformation of $(X,0)$, one can construct a semi-universal deformation of the resolution $\mathcal{F}^\ubul$. This new point of view allows Donin to provide a semi-universal deformation of $(X,0)$ (see \cite[Theorem 3.1]{Don72}) by making use of the theory of Banach analytic spaces in \cite{Dou64}.

Taking into account the presence of the actions, Ferrer and Puerta \cite{FP81, Pue82} used the same stategy to prove the existence of $\mathbb{C}^*$-equivariant deformations of $(X,0)$. One of the crucial steps is to prove the existence of a $\mathbb{C}^*$-equivariant free resolution of finite length of $\mathcal{O}_{X,0}$ - the first crucial step in Donin's construction (remark that the structural sheaf $\mathcal{O}_{X,0}$ inherits a natural $G$-action induced from the one on $(X,0)$).

In this paper,  we shall in fact follow this approach.  The main idea is quite natural and standard : we deal with compact Lie group actions first and then pass to general reductive complex Lie groups by means of complexifications. As expected, the first result is the existence of equivariant free resolutions of the structural sheaf for compact groups.
\begin{thrm} \label{Gres} Let $(X,0) \hookrightarrow (\mathbb{C}^n,0)$ be an inclusion of germs of complex spaces. Suppose that $G$ be a compact complex Lie group acting holomorphically on $(\mathbb{C}^n,0)$ leaving $(X,0)$ invariant. Then there exists a $G$-equivariant resolution of the structural sheaf $\mathcal{O}_{X,0}$
$$\mathcal{F}^\ubul: 0 \longrightarrow \mathcal{F}^r \overset{\varphi_r}{\longrightarrow}\mathcal{F}^{r-1} \overset{\varphi_{r-1}}{\longrightarrow}\cdots \overset{\varphi_{1}}{\longrightarrow} \mathcal{F}^0  \overset{\pi}{\longrightarrow} \mathcal{O}_{X,0}= \mathcal{O}_{\mathbb{C}^n,0}/I_X\longrightarrow 0$$ where $I_X$ is the ideal sheaf defining the germ $(X,0)$ and $\mathcal{F}_i \cong \mathcal{O}_{\mathbb{C}^n,0}^{\oplus m_i}$ for some $m_i \in \mathbb{N}$.
\end{thrm}
One of the applications is a $G$-equivariant version of Donin's correspondence given in \cite[\S1]{Don72} between deformations of the germ $(X,0)$ (cf. Definition \ref{defger} and Definition \ref{defdefequ}) and those of the free resolution of $\mathcal{O}_{X,0}$.
\begin{thrm}\label{gcordef} Let  $(X,0) \hookrightarrow (\mathbb{C}^n,0)$ be an inclusion of germs of complex spaces. Suppose that $G$ is a Lie compact group acting holomorphically and linearly on $(\mathbb{C}^n,0)$ leaving $(X,0)$ invariant and  $\cF^\ubul$ is a fixed $G$-equivariant free resolution of the structure sheaf $\cO_{X,0}$. Then there is a correspondence between 
\[
\{ G\text{-equivariant deformations of } (X,0) \} \;\longleftrightarrow\; \{ G\text{-equivariant deformations of  } \cF^\ubul \}.
\]
\end{thrm}
Therefore, by integrating the $G$-action along Donin's construction and using nice properties of representations of compact Lie groups, we obtain the following existence result.
\begin{thrm} \label{thm: compact}
Let $(X,0)$ be a germ of complex spaces with a distinguished point  $0$ such that $l:=\dim_{\mathbb{C}} T^1_{(X,0)} < +\infty$ and $G$ a compact complex Lie group acting holomorphically on $(X,0)$. Then there exists a $G$-equivariant semi-universal deformation $\pi: (\mathcal{X},0) \rightarrow (S,0)$ of $(X,0)$ such that $\dim_{\mathbb{C}} T_0S =l$, unique up to $G$-equivariant isomorphisms. 
\end{thrm} Here, $T^1_{(X,0)}$ is the set of isomorphism classes of deformations of $(X,0)$ over the double point $\mathbb{D}:=\spec(\bC[\epsilon])$ with $\epsilon^2=0$, which actually has a structure of $\mathbb{C}$-vector spaces (cf. Theorem \ref{def: t1} and the paragraph right before it). As a sequence,  by means of complexification, we obtain an answer to the problem in the general case.
\begin{thrm} \label{thm: reductive}
Let $(X,0)$ be a germ of complex spaces with a distinguished point  $0$ such that $l:=\dim_{\mathbb{C}}T^1_{(X,0)} < +\infty$. Let $G$ be a reductive complex Lie group and $\widetilde G$ be a connected real maximal compact subgroup of $G$ such that its complexification is nothing but $G$. Suppose that there exists a holomorphic $(G,\widetilde G)$-action on $(X,0)$. Then there exists a holomorphic $(G,\tilde G)$-equivariant semi-universal deformation $\pi: (\mathcal{X},0) \rightarrow (S,0)$ of $(X,0)$ such that $\dim_{\mathbb{C}} T_0S =l$, unique up to $(G,\tilde{G}$-equivariant isomorphisms. 
\end{thrm}
A remark is order. One can not hope to have an action of non-compact groups on germs of complex spaces, in general. For example, the natural $\mathbb{C}^*$-action on $\mathbb{C}$ is global but it induces only a local $\mathbb{C}^*$ (defined in a neighborhood of the identity or a little bit better in a neighborhood of its compact subgroup $S^1$) on the germ $(\mathbb{C},0)$. This is the reason why we use the notion \textit{local holomorphic} $(G,\tilde{G})$-\textit{action} (cf. Definition \ref{locholact}).

If the semi-universal deformation has the \textit{economy property} that sufficiently close nearby fibers can not be isomorphic to the central fiber (cf. \cite[Theorem 4.8.4]{Tei78} or \cite[Theorem 1.19]{Gre20}) then the $G$-equivariant deformation permits to detect isomorphic nearby fibers.
\begin{cor} \label{quotientmod}
Let $(X,0)$ be a germ of complex spaces with a distinguished point  $0$ such that $l:=\dim_{\mathbb{C}}T^1_{(X,0)} < +\infty$. Let $G$ be a reductive complex Lie group and $\widetilde G$ be a connected real maximal compact subgroup of $G$ such that its complexification is nothing but $G$. Suppose that there exists a holomorphic $(G,\widetilde G)$-action on $(X,0)$. Let $\pi: (\mathcal{X},0) \rightarrow (S,0)$ be the $(G,\tilde G)$-equivariant semi-universal deformation  of $(X,0)$. If $x_1,x_2 \in (\mathcal{X},0)\setminus \lbrace 0\rbrace$ sufficiently close to distinguished poit $0$ of $\mathcal{X}$ such that $x_1=g_2x_2$ for some $g \in G$ then the two fibers $(\pi^{-1}(\pi(x_1)),x_1)$ and $(\pi^{-1}(\pi(x_2)),x_2)$ are isomorphic to each other.
\end{cor}
This simple corollary has a meaning in the sense of moduli. More precisely, while the base space $(S,0)$ of the semi-universal deformatiom is often considered the best approximation  of the local moduli of $(X,0)$, it is still far from being the local moduli mainly due to the existence of nontrivial automorphisms of $(X,0)$. Besides, it is even worse that automorphism groups of germs of complex spaces (even with isolated singularities) are infinite-dimensional in general (cf. \cite{Mul87}).  Corollary \ref{quotientmod} however says heuristically that one could get even closer to the local moduli by lifting reductive $(G,\tilde{G})$-actions on $(S,0)$ and then taking the ``quotient" $(S/(G,\tilde{G}),0)$. This kind of results has been confirmed in terms of stacks algebraically by \cite{AHR20} and analytically in some particular cases by \cite{Doa24}, leaving the general case as an open problem.

One can also introduce deformations of pairs (germs, Lie groups) (cf. Definition \ref{defdefpair}) in parallel with \cite{Cat78} in which pairs are (compact complex spaces, Lie groups). As immediate consequences of Theorem \ref{thm: compact} and Theorem \ref{thm: reductive}, we obtain the following two existence resutls.
\begin{thrm} \label{thm: compactpai}
Let $(X,0)$ be a germ of complex spaces with a distinguished point  $0$ such that $\dim_{\mathbb{C}} T^1_{(X,0)} < +\infty$ and $G$ a compact complex Lie group acting holomorphically on $(X,0)$. Then there exists a semi-universal deformation $\pi: (\mathcal{X},0) \rightarrow (S^G,0)$ of $((X,0),G)$ such that $T_0S^G\cong (T^1_{(X,0)})^G$.
\end{thrm}
\begin{thrm} \label{thm: reductivepai}
Let $(X,0)$ be a germ of complex spaces with a distinguished point  $0$ such that $\dim_{\mathbb{C}}T^1_{(X,0)} < +\infty$. Let $G$ be a reductive complex Lie group and $\widetilde G$ be a connected real maximal compact subgroup of $G$ such that its complexification is nothing but $G$. Suppose that there exists a holomorphic $(G,\widetilde G)$-action on $(X,0)$. Then there exists a holomorphic semi-universal deformation $\pi: (\mathcal{X},0) \rightarrow (S^{(G,\tilde G)},0)$ of $((X,0),(G,\tilde G))$.
\end{thrm}
In general, it is extremely hard to construct moduli spaces of a given type of objects due to the presence of symmetries as mentioned before.  Nevertheless,  there are two essential ways to rigidify the problem : the first one is to require the symmetries to fix more points and the second one is to classify them in the presence of symmetries. It is clear that Theorem \ref{thm: compactpai} and Theorem \ref{thm: reductivepai} fall into the second category. As more one simple consequence, one has the following rigidity result (cf. Definition \ref{defdefpairrig} for rigidity in comparison with the usual rigidity notion in \cite[Definition 2.1.1]{Gre20}), similar to a partial rigidity criterion given in \cite[Proposition 4]{Cat78}.
\begin{cor} \label{thm: reductivepairig}
Let $(X,0)$ be a germ of complex spaces with a distinguished point  $0$ such that $\dim_{\mathbb{C}}T^1_{(X,0)} < +\infty$. Let $G$ be a reductive complex Lie group and $\widetilde G$ be a connected real maximal compact subgroup of $G$ such that its complexification is nothing but $G$. Suppose that there exists a holomorphic $(G,\widetilde G)$-action on $(X,0)$. Then $T^1_{(X,0)} $ inherits a natural $\tilde{G}$-action. Furthermore, if  $ (T^1_{(X,0)})^{\tilde{G}}=\lbrace 0 \rbrace$ then the pair  $((X,0),(G,\tilde G))$ is rigid.
\end{cor} In the above corollary, if we take $G$ to be a finite group and form the quotient germ $(X/G,0)$ together with the natural map $(X,0) \rightarrow (X/G,0)$, then the rigidity of $(X/G,0)$ can be described in terms of that of the pair $((X,0),G)$. For example, it follows from \cite[Lemma 8]{Ste88} that if  $(X,0)$ is Cohen-Macaulay and the natural map $\sigma: (X,0) \rightarrow (X/G,0)$ is \'{e}tale in codimension two then the pair $((X,0),G)$ is rigid if and only if the germ $(X/G,0)$ is rigid. It happens quite often that computing $(T^1_{(X,0)})^G$ is less difficult than computing $T^1_{(X/G,0)}$. For example, we can consider $X=\lbrace x^2+y^2+z^2=0 \rbrace \subset \mathbb{C}^3$ with a $\mathbb{Z}_2$-action by reflections.

One also has a criterion for group action extensions through deformations, similar to \cite[Proposition 5]{Cat78}.
\begin{cor} \label{thm: compactpaipro}
Let $(X,0)$ be a germ of complex spaces such that $\dim_{\mathbb{C}} T^1_{(X,0)} < +\infty$ and $G$ a compact complex Lie group acting holomorphically on $(X,0)$. If $G$ acts trivially on $T^1_{(X,0)}$ and the base $(S^G,0)$ is smooth then the gems $(S,0)$ and $(S^G,0)$ coincide. In particular, for any deformation $(\mathcal{Y},0) \rightarrow (T,0)$, the $G$-action on $(X,0)$ extends to an $G$-action on $(\mathcal{Y},0)$.
\end{cor}
The next contribution of this paper is to show that the reductivity assumption in this story is optimal. Namely, we provide a simple germ of complex spaces together with a local holomorphic action of a non-reductive group such that this action can not be equivariantly extended to its semi-universal deformation (cf. Section \ref{sec: op} for more details). Thus, it is consistent with the case of algebraic varieties and compact complex manifolds treated in \cite{Doa20}. Hence, together with the main results, it gives another vivid illustration for the philosophy : ``\textit{Reductive subgroups of the automorphism group of the analytic object
under deformation can be (at least locally) analytically extended to its semiuniversal deformation}" given at the end of \cite{Doa22}.

Theorem \ref{thm: compact} also has a simple application to actions on Milnor fiber. Namely, let $f: (\mathbb{C}^{n+1},0)\rightarrow (\mathbb{C},0)$ is a germ of complex spaces $(X,0)$ with isolated singularity at $0$, it is well-known that there exists a local fibration over $\mathbb{C} \setminus \lbrace 0 \rbrace$ with fiber $F$, called Milnor fibration. If $G$ is a finite group acting linearly on $(\mathbb{C}^{n+1},0)$ leaving $f$ invariant then $G$ acts naturally on the cohomology $H^n(F,\mathbb{C})$ and on the semi-universal space $S$ of $(X,0)$ (which is not obvious since it actually follows from  Slodowy's result on $G$-equivariant semi-universal deformations for $G$ a finite group, given in \cite{Slo78}) and eventually on its structural sheaf $\mathcal{O}_{S}$. Note also that $G$ preserves the K\"{a}hler form $\omega =dx_0\wedge\ldots\wedge dx_{n+1}$ on $\mathbb{C}^{n+1}$, hence gives rise to a character $\lambda: G \rightarrow \mathbb{C}^*$ given by the determinant. There are several results relating $H^n(F,\mathbb{C})$ and $\mathcal{O}_S$ as $G$-modules : \cite{Arn78} for $G$ a single reflexion; \cite{OS78a} for $f$ a weighted homogeneous function with applications to the computation of the character and the eigenvalues of the $G$-action in \cite{OS78b}. The most general result so far is due to Wall stating that if $G$ is finite and $f$ is an arbitrary hypersuface singularity then $H^n(F,\mathbb{C})$ is isomorphic to $\mathcal{O}_S\otimes_{\mathbb{C}}\lambda$ as $G$-modules (see \cite[Theorem]{Wal80a}). 

Here, we generalize Wall's theorem to the case that $G$ is a compact Lie group acting holomorphically on $(\mathbb{C}^{n+1},0)$ leaving $f$ invariant. More specially, since $G$ is compact, its action on $\mathbb{C}^{n+1}$ can be linearized after a change of coordinates by Bochner's linearization theorem (cf. \cite{Kau65, Akh95}). As a sequence, one still gets a character $\lambda: G \rightarrow \mathbb{C}^*$. In addition, for the sake of Theorem \ref{thm: compact}, $G$ also acts on the semi-universal space $S$ and then on its structural sheaf $\mathcal{O}_S$.
\begin{thrm} \label{thm: mil} With the above setting, if $G$ is a complex compact Lie group and $f$ is an arbitrary isolated hypersurface singularity, there exists a natural isomorphism between $H^n(F,\mathbb{C})$ and $\mathcal{O}_S\otimes_{\mathbb{C}}\lambda$ as $G$-modules.
\end{thrm}
 
To conclude this subsection, we should mention that there is a much more general approach given in \cite{BK87}, which can be used to construct semi-universal deformations of almost every analytic object of interest : holomorphic maps, principal bundles, coherent analytic sheaves, compact complex manifolds, compact complex spaces, isolated singularities, etc (cf. \cite[Section 4]{BK87}) as immediate corollaries. A brief review of it is given in \cite{Bin87} where the author even claimed that the method could be made equivariant and that the details would be given in subsequent papers \cite[Remark 4.8.(3)]{Bin87}, which seem to have never appeared since then. If it was indeed the case, one would have a unified theory of equivariant deformations of analytic objects. Hence, we really hope to be able to address this problem in future work.
\begin{problem} Can the general formalism in \cite{BK87} be made $G$-equivariant for $G$ a reductive complex Lie group ?
\end{problem}
\subsection{Formal deformation theory and Grauert's approximation}
Formal deformation theory was systematically studied by M. Schlesssinger \cite{Sch68} where the central result is  a criterion for a functor of artinian rings to admit a formal semi-universal element. One application is the existence of formal semi-universal deformations of affine schemes with at most isolated singularities or projective schemes. Rim reconsidered the theory of functors of artinian rings in the presence of group actions which come mainly from the automorphism group of the objects under deformation. He showed that if the group is linearly reductive, we can even put equivariant structures on formal semi-universal elements (cf. \cite{Rim80}), generalizing Schlessinger's results. In particular, for projective schemes or affine schemes with at most isolated singularities with linearly reductive automorphism groups, the formal semi-universal deformations can be made formally equivariant. In \cite[Section 3]{Slo01}, Slodowy  also mentioned that \cite{Rim80} should already give equivariant semi-universal deformations in the case of germs of algebraic varieties. However, it is apparently only a formal result where the convergence has not been justified yet, to the author's knowledge. This is also a crucial difference between the algebraic world and the analytic one. Thus, the problem in the introduction is still open even for algebraic singularities in general. 

Returning to our main problem, given a germ of complex spaces $(X,0)$, we can consider the functor of artinian rings associated to the deformation problem of $(X,0)$. Following once again Schlessinger's formalism, it can be checked  that this functor admits a formal semi-universal element, i.e. a formal semi-universal deformation (cf. \cite[Theorem 10.3.16]{JP00}) under the finite-dimensional tangency condition. In terms of analytic equations, it means that we have a formal solution. The next step is to use Grauert's powerful approximation theorem (see Theorem \ref{grauapp}) - an upgraded version of Artin's approximation theorem (cf. Theorem \ref{artappr}) to produce a convergent solution, approximating the formal one and leaving the semi-universality untouched (cf. \cite{Gra72} or \cite[Theorem 10.3.21]{JP00}). This gives rise to a semi-universal deformation of $(X,0)$. Recently, Greuel and Pfister have generalized Grauert's approximation theorem to convergent power series over arbitrary real valued fields of any characteristic. Again, a direct consequence is the existence of semi-universal deformations of isolated singularities, as  expected (cf. \cite{GL25b}).

Finally, we let the linearly reductive group $G$ rejoin the game. Inspired by the above construction, in order to respond to our question, one can verify that Rim's formalism \cite{Rim80} can be used to produce a $G$-equivariant formal solution.  The natural idea now is to prove an appropriate $G$-equivariant version of Grauert's approximation theorem - Theorem \ref{grauapp} and then to use it to produce a $G$-equivariant convergent solution approximating the formal one. To this end, one could hope to imitate the details of \cite{BM79} where the authors proved an equivariant version of Artin's approximation theorem  (cf. Theorem \ref{equiart}). The main idea could be to take the average with respect to the group action which is possible if the group is compact (since it implies that Haar measures are available) and then pass to reductive groups by complexification. 
All of these motivate us to pose the following question.
\begin{question} Does there exist a $G$-equivariant version of Grauert's approximation theorem where $G$ is a reductive complex Lie group ? 

\end{question}

\subsection*{Organization} In Section \ref{sec: defger}, we recall the theory of classical deformations of germs of complex spaces and its fundamental results. Next, we define and study equivariant deformations in general in Section \ref{sec: defgerequ} and over the double point in particular in Section \ref{sec: doupoi}. After that, the existence of equivariant semi-universal deformations and its consequences are given in Section \ref{sec: exi}. Several examples are provided in Section \ref{sec: op} in order to justify the reductivity assumption. Finally, Appendix \ref{sec: anaequapp} is devoted to a brief summary of the approximation theorems mentioned in the introduction for the reader's convenience.

\begin{acknowledgement} The author would like to warmly thank Prof. S. Kosarew for explaining some basic properties of Banach analytic spaces and for pointing out the reference \cite{BK87} to him. This research is funded by a start-up research fund from the Harbin Institute of Technology.
\end{acknowledgement}

\section{Deformations of germs of complex spaces} \label{sec: defger}In this section, we recall basic definitions and fundamental results on the theory of deformations of germs of complex spaces. The curious reader is referred to \cite{JP00, GLS25} for a complete treatment of the subject.
\begin{defn} \label{defger}$\mathrm{(}$\cite[Definition II.1.1]{GLS25}$\mathrm{)}$ Let $(X,0)$ and $(S,0)$ be germs of complex spaces. A \textit{deformation of} $(X,0)$ \textit{over} $(S,0)$ is a flat morphism $\pi:(\cX,0)\rightarrow (S,0)$ of germs of complex spaces such that we have an isomorphism between $(\pi^{-1}(0),0)$ and $(X,0)$. In other words, the following diagram 
\[
\begin{tikzcd}
(X,0) \arrow[r,hook, "\iota"] \arrow[d]
& (\mathcal X, 0) \arrow[d, "\pi"] \\[4pt]
\{0\}\arrow[r, hook]
& (S,0)
\arrow[ul, phantom, "\square", very near start, pos=0.52]
\end{tikzcd}
\]
is Cartesian.
\end{defn}

\begin{defn} $\mathrm{(}$\cite[Definition II.1.2]{GLS25}$\mathrm{)}$ Let $\pi_1:(\cX_1,0)\rightarrow (S_1,0)$ and $\pi_2:(\cX_2,0)\rightarrow (S_2,0)$ be two deformations of $(X,0)$. A \textit{morphism} between them is a pair $(\Phi,\phi)$ of two morphisms of germs of complex spaces  such that the following diagram
\[
\begin{tikzcd}[row sep=1.5cm]
& (\cX_1,0) \arrow[r, "\Phi"] \arrow[d,"\pi_1" swap, pos=0.45]
& (\cX_2,0) \arrow[d, "\pi_2" swap, pos=0.45] \\
(X,0) \arrow[ru, hook, "\iota_1"] \arrow[d]\arrow[rru,hook, "\iota_2" swap, pos=0.65]
& (S_1,0) \arrow[r,"\phi" ] 
& (S_2,0) \\
\{0\} \arrow[ru, hook]\arrow[rru, hook]
& 
& 
\end{tikzcd}
\]
is commutative. In particular, if $(S_1,0)=(S_2,0)=(S,0)$ and $\phi= \mathrm{Id}$ and $\Phi$ is an isomorphism, then we say that the two given deformations over the same base $(S,0)$ are \textit{isomorphic}.
\end{defn}
\begin{rmk} If there is a morphism $(\Phi,\phi)$ between deformations of $(X,0)$ over the same base $(S,0)$, then $\Phi$ is  an isomorphism due to for example \cite[Lemma I.1.86]{GLS25}. Hence, the given two deformations are isomorphic to each other.
\end{rmk}
The following result provides a description of deformations of $(X,0)$ in terms of generators and relations.  This allows to facilitate the study of (equivariant) deformations, which we are going to see shortly.
\begin{lemma}$\mathrm{(}$\cite[Corollary II.1.6]{GLS25}$\mathrm{)}$ \label{defemb} If $(X,0)$ is given as subgerm of $(\bC^n,0)$ for some $n$ then any deformation $\pi:(\cX,0)\rightarrow (S,0)$ of $(X,0)$ can be embedded. In other words, the following diagram 

\begin{center}
\begin{tikzpicture}[every node/.style={anchor=center}, scale=0.85, transform shape]
\node (X0) at (0,4) {$(X,0)$};
\node (Xc) at (3,4) {$(\mathcal X,0)$};
\node (Cn) at (0,2) {$(\mathbb C^n,0)$};
\node (CxS) at (3,2) {$(\mathbb C^n,0)\times(S,0)$};
\node (Z0) at (0,0) {$\{0\}$};
\node (S0) at (3,0) {$(S,0)$};

\draw[->] (X0) -- (Xc); 
\draw[->] (Cn) -- (CxS); 
\draw[->] (Z0) -- (S0);

\draw[->] (X0) -- (Cn) node[midway,left] {$\iota_0$}; 
\draw[->] (Xc) -- (CxS) node[midway,right] {$\iota$}; 
\draw[->] (Cn) -- (Z0); 
\draw[->] (CxS) -- (S0) node[midway,right] {$\mathrm{pr}_2$}; 

\node at ($(X0)!0.5!(CxS)$) {$\square$};
\node at ($(Cn)!0.5!(S0)$) {$\square$};

\draw[->, bend left=70] (Xc.south east) to node[right,pos=0.5] {$\pi$} (S0.north east);
\end{tikzpicture}
\end{center}
is Cartesian.

\end{lemma}
In this way, the set of isomorphism classes of deformations of $(X,0)$ over the double point $\mathbb{D}:=\spec(\bC[\epsilon])$ with $\epsilon^2=0$, denoted by $T^1_{(X,0)}$ can be described explicitly as follows. Let $I=\lbrace f_1,\ldots,f_k \rbrace$ be the defining ideal of $(X,0)$ in $(\bC^n,0)$ where $f_i \in \cO_{\bC^n,0}$ and let $\Theta_{\bC^n,0}$ be the free $\cO_{\bC^n,0}$-module of derivarions of $\cO_{\bC^n,0}$. Consider the natural map \[
\begin{aligned}
\alpha : \Theta_{\bC^n,0}&\longrightarrow \Hom_{\cO_{\bC^n,0}}(I,\cO_{X,0}) \\
\theta &\longmapsto (f \mapsto \theta(f)).
\end{aligned}
\]
\begin{thrm} \label{def: t1} $\mathrm{(}$\cite[Lemma 10.2.7 and Theorem 10.2.13]{JP00}, \cite[Propsition II.1.25]{GLS25}$\mathrm{)}$ Let $(X,0) \subset (\bC^n,0)$ defined by an ideal $I\subset \cO_{\bC^n,0}$. Then
\begin{enumerate}
\item[$\mathrm{(i)}$] deformations of $(X,0)$ over the double point $\mathbb{D}$ are in bijective correspondence with elements of $\Hom_{\cO_{\bC^n,0}}(I,\cO_{X,0})$;
\item[$\mathrm{(ii)}$] $T^1_{(X,0)}\cong \mathrm{coker}(\alpha)$.
\end{enumerate}
\end{thrm}

In the category of germs of complex spaces, one still can  define the \textit{fiber product} of two morphisms of germ of complex spaces (cf. \cite[Definition I.1.46 and Page 58]{GLS25}). In particular, for $\pi:(\cX,0)\rightarrow (S,0)$  a deformation of $(X,0)$ and $\phi: (T,0) \rightarrow (S,0)$ a morphism of germs of complex spaces, the fiber product $(\cX,0)\times_{(S,0)} (T,0)$ together with the natural second projection $$(\cX,0)\times_{(S,0)} (T,0) \rightarrow (T,0)$$ which is also flat, determines a deformation of $(X,0)$ over $(T,0)$.  This is often called the \textit{pull-back} or the \textit{induced deformation} by the morphism $\phi$ (cf. \cite[Definition II.2.3]{GLS25})
\begin{defn} \label{defsemuni} A deformation  $\pi:(\cX,0)\rightarrow (S,0)$ is said to be \textit{semi-universal} if for any other deformation  $\rho:(\cY,0)\rightarrow (T,0)$, there exists a morphism of germs of complex spaces $\phi: (T,0) \rightarrow (S,0)$ whose differential is unique  such that $\rho:(\cY,0)\rightarrow (T,0)$ is isomorphic to the pullback $(\cX,0)\times_{(S,0)} (T,0) \rightarrow (T,0)$ by $\phi$, i.e. the following diagram is Cartesian
\[
\begin{tikzcd}
(\mathcal{Y},0) \arrow[r, "\exists \Phi"] \arrow[d]
& (\mathcal X, 0) \arrow[d, "\pi"] \\[4pt]
(T,0)\arrow[r, "\exists \phi"]
& (S,0)
\arrow[ul, phantom, "\square", very near start, pos=0.52]
\end{tikzcd}
\] If the morphism $\phi$ is further unique then the deformation  $\pi:(\cX,0)\rightarrow (S,0)$ is said to be \textit{universal}. 
\end{defn}
By definition, a semi-universal deformation if exists is unique up to non-canonical isomorphisms. The following result is fundamental in the theory of deformations of germs of complex spaces.
\begin{thrm} $\mathrm{(}$\cite{Don72, Gra72}$\mathrm{)}$ Let $(X,0)$ be a germ of complex spaces. Suppose that $\dim_{\bC}T^1_{(X,0)} < + \infty$. Then there exists a semi-universal deformation of $(X,0)$.
\end{thrm}
\begin{rmk} Note that the condition that  $\dim_{\bC}T^1_{(X,0)} < + \infty$ is actually necessary and sufficient for the existence of semi-universal deformation (cf. \cite{Gre20}).  It is well-known that if $(X,0)$ is a germ of complex spaces with an isolated singularity then $\dim_{\bC}T^1_{(X,0)}$ is a finite-dimensional $\bC$-vector space (cf. \cite[Corollary 1.34]{Gre20} or \cite[Theorem 10.2.15]{JP00} for example).
\end{rmk}
\begin{rmk} One can also define the space $T^2_{(X,0)}$ of \textit{obstructions to deformation extensions}. In particular, if $T^2_{(X,0)}$ is trivial, then the base of the semi-universal deformation is smooth (cf. \cite[Definition 10.3.5]{JP00} or \cite[Proposition II.1.29]{GLS25}.)
\end{rmk}
\section{Equivariant deformations of germs of complex spaces} \label{sec: defgerequ}
This section is devoted to developing a theory of equivariant deformations. Firstly, one needs to specify what kind of group actions is the right one to investigate. Since we are working with  germs of complex spaces, the usual global definition of group actions shall not work especially when the group is not compact. In order to remedy this, we are going to use a local definition of group actions on germs of complex spaces, already introduced in \cite{Doa21}. Namely, let $(X,0)$ be a germ of complex spaces and $G$ be a complex Lie group. We define the so-called \textit{local} $(G,K)$-\textit{action}  $(X,0)$ where $K$ is a compact subgroup of $G$. Denote by $\prod_X$ the collection of all pair $(U_\pi,V_\pi)$, where $U_\pi$ and $V_\pi$ are open neighborhoods of $0$ in $X$ such that $U_\pi\Subset V_\pi$. Suppose that for each $\pi \in \prod_X$ we have an open neighborhood $G_\pi$ of $K$ and a mapping $\Phi_\pi:\; G_\pi \rightarrow \Ho(U_\pi,V_\pi)$ where $\Ho(U_\pi,V_\pi)$ is the set of all holomorphic functions from $U_\pi$ to $V_\pi$.
\begin{defn} \label{locholact} $\mathrm{(}$\cite[Definition 5.1]{Doa21}$\mathrm{)}$ One says that the system $\lbrace \Phi_\pi\rbrace$ defines a \textit{local} $(G,K)$-\textit{action} (or  $(G,K)$-\textit{action} for short) on $X$ if the following conditions are satisfied.
\begin{enumerate}
\item[(a)] For all $g,h \in G$ such that $k:=gh\in G_\pi$, we have
$$\Phi_{\pi}(g)\circ \Phi_\pi(h)\mid_{U_{\pi,h}}= \Phi_\pi(k)\mid_{U_{\pi,h}}$$ where $U_{\pi,h}:=\lbrace x \in U_\pi \mid \Phi_\pi(h)(x)\in U_\pi \rbrace $;
\item[(b)] $\Phi_\pi(\mathbf{1}_G)=\mathbf{id}$;
\item[(c)] for all $\pi, \rho \in \prod_X$ and $g\in G_\pi\cap G_\rho$ we have
$$\Phi_\pi(g)\mid_{U_{\pi}\cap U_{\rho}}=\Phi_\rho(g)\mid_{U_{\pi}\cap U_{\rho}} $$ so that $gx:=\Phi_\pi(g)x$ is independent of the choice of $\pi$ with $x\in U_\pi,g\in G_\pi$;
\item[(d)] for any two open sets $U\Subset U_\pi$ and $V\Subset V_\pi$, the set 
$$W:=W_{\overline{U},V}:=\lbrace g\in G_\pi \mid g\cdot \overline{U}\subset V \rbrace $$ is open in $G_\pi$ and the map \begin{align*}
*:W&\rightarrow \mathcal{O}(U)\\
g &\mapsto f\circ g\mid_U
\end{align*} is continuous for all $f\in \mathcal{O}(V)$ where $\overline{U}$ is the closure of $U$ and $\mathcal{O}(P)$ is the set of holomorphic functions on $P$ for any open subset $P$ containing $0$ of $X$;
\item[(e)] The restriction of the system $\lbrace \Phi_\pi\rbrace$ on $K$ gives a global $K$-action on $X$, i.e. a homomorphism of topological groups $\Phi:\; K\rightarrow \Au(X,0)$.
\end{enumerate}
Moreover, if $G$ is a real (resp. complex) Lie group and if  $*$ and $\Phi$ are real analytic (resp. holomorphic), then the local $(G,K)$-action is called real analytic (resp. holomorphic). Two local $(G,K)$-actions defined by two systems $\lbrace \Phi_\pi\rbrace$ and $\lbrace \Phi'_\pi\rbrace$  are said to be equivalent if for all $\pi\in \prod_X$, the mappings $\Phi_\pi:\; G_\pi \rightarrow \Ho(U_\pi,V_\pi)$ and $\Phi'_\pi:\; G'_\pi \rightarrow \Ho(U_\pi,V_\pi)$ coincide on a sub-domain $G_\pi \cap G'_\pi$ containing $K$ and their restrictions on $K$ give the same global $K$-action. 
\end{defn}
\begin{rmk} In the above definition,
\begin{enumerate}
\item[(i)] of course any local $(G,K)$-action is a \textit{local} $G$-\textit{action} and in particular if $K= \{ \mathrm{Id}_G\}$, then one recovers a \textit{local} $G$-\textit{action} on germs of complex spaces in the usual sense (cf. \cite[Section 1.2]{Akh95});
\item[(ii)] if $K= \{ \mathrm{Id}_G\}$ and $G$ is compact then one recovers the usual global $G$-action on $(X,0)$. This justifies why Condition (e) is pertinent. In this case, we shall omit ``global" and say $G$-action for simplicity.
\end{enumerate}
\end{rmk}
Now, we are in a position to give a definition of equivariant deformations of a germ of complex spaces $(X,0)$.
\begin{defn} \label{defdefequ}
 Let $G$ be a complex Lie group and $K$ be a compact subgroup of $G$ which determines a holomorphic $(G,K)$-action on $(X,0)$. A real analytic (resp. holomorphic) $(G,K)$-\textit{equivariant deformation} of $X$ is a usual deformation of $(X,0)$ $\pi$: $(X,0)\rightarrow (S,0)$ equipped with a real analytic (resp. holomorphic) $(G,K)$-action on $X$ extending the given (resp. holomorphic) $(G,K)$-action on $(X,0)$ and a real analytic (resp. holomorphic) $(G,K)$-action on $(S,0)$ with respect to which $\pi$ is a $G$-equivariant map. We call these extended actions a real analytic (resp. holomorphic) $(G,K)$-equivariant structure on $\pi$: $(X,0)\rightarrow (S,0)$.
\end{defn}
\begin{rmk}
For simplicity, by $(G,K)$-actions (resp. $(G,K)$-equivariant deformations), we really mean real analytic $(G,K)$-actions (resp. real analytic $(G,K)$-equivariant deformations).
\end{rmk}
\begin{defn} A deformation  $\pi:(\cX,0)\rightarrow (S,0)$ is said to be $(G,K)$-\textit{equivariantly semi-universal} if it is semi-universal in sense of Definition \ref{defsemuni} and   if for any other $(G,K)$-equivariant  deformation  $\rho:(\cY,0)\rightarrow (T,0)$, there exists a $(G,K)$-equivariant morphism of germs of complex spaces $\phi: (T,0) \rightarrow (S,0)$ whose differential is unique  such that $\rho:(\cY,0)\rightarrow (T,0)$ is $(G,K)$-equivariantly isomorphic to the pullback $(\cX,0)\times_{(S,0)} (T,0) \rightarrow (T,0)$ by $\phi$. In other world, the following Cartesian diagram
\[
\begin{tikzcd}
(\mathcal{Y},0) \arrow[r, "\exists \Phi"] \arrow[d]
& (\mathcal X, 0) \arrow[d, "\pi"] \\[4pt]
(T,0)\arrow[r, "\exists \phi"]
& (S,0)
\arrow[ul, phantom, "\square", very near start, pos=0.52]
\end{tikzcd}
\] can be made further $(G,K)$-equivariant.
\end{defn}
 Therefore, we can rephrase our objective as finding a real analytic (resp. holomorphic) $(G,K)$-equivariant semi-universal deformation of a given germs of complex spaces with a real analytic (resp. holomorphic) $(G,K)$-action.

\begin{rmk} The main source of group actions on germs of complex spaces should already come from their automorphism group. Symmetry of singularities was investigated by many authors in the eighties where one of the most important results is the existence of maximal reductive subgroups. The interested reader is referred to  \cite{Orl79, Wal80, Wah83, HM89} for great details. Non-reductive subgroups are also abundant (cf. \cite[Theorem 5]{Mul87}). 
\end{rmk}

According to \cite[Section 1]{Don72}, a deformation of $(X,0)$ can be thought of as a deformation of a fixed resolution of the structure sheaf $\cO_{X,0}$ 
\begin{equation}\label{resfix}\mathcal{F}^\ubul: 0 \longrightarrow \mathcal{F}^r \overset{\varphi_r}{\longrightarrow}\mathcal{F}^{r-1} \overset{\varphi_{r-1}}{\longrightarrow}\cdots \overset{\varphi_{1}}{\longrightarrow} \mathcal{F}^0  \overset{\pi}{\longrightarrow} \mathcal{O}_{X,0}= \mathcal{O}_{\mathbb{C}^n,0}/\mathcal{I}_X\longrightarrow 0
\end{equation} where $I_X$ is the ideal sheaf defining the germ $(X,0)$ and $\mathcal{F}_i \cong \mathcal{O}_{\mathbb{C}^n,0}^{\oplus m_i}$ for some $m_i \in \mathbb{N}$. More precisely, in view of Lemma \ref{defemb}, if $\mathcal{X}\rightarrow (S,0)$ be a deformation of $(X,0)$ given in the product $(\mathbb{C}^n,0)\times (S,0)$ then the resolution $\mathcal{F}^{\ubul}\rightarrow \mathcal{O}_X$ (\ref{resfix}) can be extended to a resolution $\widetilde{ \mathcal{F}}^{\ubul} \rightarrow \mathcal{O}_{\mathcal{X}}$ of the structure sheaf $\mathcal{O}_{\mathcal{X}}$ where each $\tilde{\mathcal F}^{i}$ is a free $\mathcal{O}_{(\mathbb{C}^n,0)\times (S,0)}$-module. Conversely, such an extension determines a deformation of $(X,0)$ given in the product $(\mathbb{C}^n,0)\times (S,0)$. Fortunately, this correspondence is still available in the equivariant context at least when the group $G$ is compact. To see this, we need the following $G$-equivariant version of Lemma \ref{defemb}.
\begin{lemma} \label{gdefemb} Let $(X,0)$ be given as subgerm of $(\bC^n,0)$ for some $n$ and let $G$ be a complex compact Lie group acting holomorphically on $(\mathbb{C}^n,0)$ leaving $(X,0)$ invariant. Then any $G$-equivariant deformation $\pi:(\cX,0)\rightarrow (S,0)$ of $(X,0)$ can be $G$-equivariantly embedded. In other words, if we equip the product $(\mathbb C^n,0)\times(S,0)$ with the diagonal $G$-action coming from the given $G$-actions on each factor, then the following diagram 

\begin{center}
\begin{tikzpicture}[every node/.style={anchor=center}, scale=0.85, transform shape]
\node (X0) at (0,4) {$(X,0)$};
\node (Xc) at (3,4) {$(\mathcal X,0)$};
\node (Cn) at (0,2) {$(\mathbb C^n,0)$};
\node (CxS) at (3,2) {$(\mathbb C^n,0)\times(S,0)$};
\node (Z0) at (0,0) {$\{0\}$};
\node (S0) at (3,0) {$(S,0)$};

\draw[->] (X0) -- (Xc); 
\draw[->] (Cn) -- (CxS); 
\draw[->] (Z0) -- (S0);

\draw[->] (X0) -- (Cn) node[midway,left] {$\iota_0$}; 
\draw[->] (Xc) -- (CxS) node[midway,right] {$\iota$}; 
\draw[->] (Cn) -- (Z0); 
\draw[->] (CxS) -- (S0) node[midway,right] {$\mathrm{pr}_2$}; 

\node at ($(X0)!0.5!(CxS)$) {$\square$};
\node at ($(Cn)!0.5!(S0)$) {$\square$};

\draw[->, bend left=70] (Xc.south east) to node[right,pos=0.5] {$\pi$} (S0.north east);
\end{tikzpicture}
\end{center}
is Cartesian where all the map appearing are further $G$-equivariant.
\end{lemma}
\begin{proof}
Since $G$ is compact, one can assume the $G$-actions are linear after a change of holomorphic coordinates (see \cite{Kau65} or \cite[Section 2.2]{Akh95}). Regardless of the $G$-actions, there exists an embedding $\tilde \iota:=(\tilde \rho,\pi)$ fitting into the above Cartesian diagram, where $\tilde \rho: (\cX,0) \rightarrow (\mathbb C^n,0)$ is a map lifting the embedding $\iota_0: (X,0) \rightarrow (\mathbb C^n,0)$ (cf. \cite[Proposition 1.5 and Lemma 1.6]{GLS25}). Let $\mu$ be a Haar measure on the compact group $G$ such that $\int_Gd\mu(g)=1$. For $x\in (\mathcal{X},0)$, since the $G$-action on the product $(\mathbb C^n,0)\times(S,0)$ is diagonal and the morphism $\pi$ is already $G$-equivariant, one can  define   \begin{align*}
\iota (x) &= \int_G g\cdot \widetilde\iota(g^{-1}\cdot p) d\mu(g)  \\
          &= \left ( \int_G g\cdot \widetilde\rho(g^{-1}\cdot x) d\mu(g) ,\pi(x)\right )\\
          &=:(\rho,\pi)           
\end{align*} It is routine to check that $\iota$ is a well-defined holomorphic map. Furthermore, its $G$-equivariance follows from the left-invariance of the measure $\mu$. The commutativity of the diagram follows from the $G$-equivariance of the other maps  and the definition of $\iota$. In particular, $\mathrm{pr}_2 \circ \iota =\pi$. Finally, observe that $\rho: (\cX,0) \rightarrow (\mathbb C^n,0)$ is a still  map lifting the embedding $\iota: (X,0) \rightarrow (\mathbb C^n,0)$ due to the $G$-equivariance of $\iota_0$. Thus, by repeating the proof of \cite[Proposition 1.5]{GLS25}, $\iota$  remains an embedding. This ends the verification.
\end{proof}

\begin{proof}[Proof of Theorem \ref{Gres}] Fix a minimal free resolution  
$$\mathcal{F}^\ubul: 0 \longrightarrow \mathcal{F}^r \overset{\varphi_r}{\longrightarrow}\mathcal{F}^{r-1} \overset{\varphi_{r-1}}{\longrightarrow}\cdots \overset{\varphi_{1}}{\longrightarrow} \mathcal{F}^0  \overset{\pi}{\longrightarrow} \mathcal{O}_{X,0}= \mathcal{O}_{\mathbb{C}^n,0}/I_X\longrightarrow 0$$ which exists by \cite{Don72}. Here, minimality means $\mathrm{im} \varphi_i \subset \mathfrak{m}\mathcal{F}^{i-1}$ where $\mathfrak{m}$ is the maximal ideal of the local ring $\mathcal{O}_{\mathbb{C}^n,0}$. Moreover, we can take $\mathcal{F}^0$ to be $\mathcal{O}_{\mathbb{C}^n,0}$. We shall replace $\mathcal{F}^\ubul$ by a new resolution where all the differentials are $G$-equivariant. This new resolution will be constructed by induction. Observe that $\pi$ is already $G$-equivariant by assumption.

Consider $\mathcal{K}_0:=I_X$. The vector space $V_0:= \mathcal{K}_0/\mathfrak{m} \mathcal{K}_0$ is finite-dimensional over $\mathbb{C}$. The $G$-action on $\mathcal{O}_{\mathbb{C}^n,0}$ preserves $\mathcal{I}_X$ and $\mathfrak{m}$ hence induces a linear representation of $G$ on $V_0$. Let $\lbrace v_1,\ldots,v_{m_0} \rbrace$ be a $\mathbb{C}$-basis of $V_0$. Choose a $\mathbb{C}$-linear section $s_0 :V_0 \rightarrow K_0 \subset  \mathcal{O}_{\mathbb{C}^n,0}$ sending for example $v_j$ to an element $s_j \in I_X$ whose class in $V_0$ is exactly $v_j$. Define an averaged section $\sigma_0 : V_0 \rightarrow K_0$ by 
$$\sigma_0(v):=\int_G g\cdot (s_0(g^{-1}\cdot v))d\mu(g).$$
where $\mu$ is a Haar measure on $G$ such that $\int_G d\mu(g)=1.$ It is easy to see that $\sigma_0(v) \in \mathcal{O}_{\mathbb{C}^n,0}$ for each $v \in V_0$. Moreover, $\sigma_0$ is $G$-equivariant. Finally, $\sigma_0$ is still a section of the projection $\mathcal{K}_0 \rightarrow V_0$ since $s_0$ is so. In this way, the image $W_0:=\sigma_0(V_0)$ is a finite-dimensional $G$-stable subspace of $K_0$ isomorphic to $V_0$ and the composition $W_0\rightarrow \mathcal{K}_0 \rightarrow V_0$ is an isomorphism. Define $\mathcal{F}^1:=\mathcal{O}_{\mathbb{C}^n,0}^{\oplus m_0}$ with canonical basis $e_1,\ldots,e_{m_0}$. Under the identification $e_j \leftrightarrow v_j$, we get a $G$-action on $\lbrace e_1,\ldots,e_{m_0}\rbrace$ as a $\mathbb{C}$-vector space. This allows us to equip $\mathcal{F}^1$ with a natural action as follows. For $f=\sum f_j e_j \in \mathcal{F}^1$, we set 
$$g\cdot f=\sum(g\cdot f_j)(g\cdot e_j)$$ where the former $\cdot$ is the usual $G$-action on $\mathcal{O}_{\mathbb{C}^n,0}$. At last, we define the first differential $$\varphi_1: F_1 \rightarrow \mathcal{O}_{\mathbb{C}^n,0},\; e_j \mapsto \sigma_0(v_j)\in W_0 \subset I_X .$$ It is a routine computation to check that $\varphi$ is $G$-equivariant. Moreover, $\sigma_0(v_j)$'s generate $\mathcal{I}_X$ since their classes in $V_0$ form a basis for $V_0$ by construction. Therefore, $\mathrm{im}( \varphi_1) =I_X$ due to Nakayama's lemma. Hence the sequence
$$\mathcal{F}^1 \overset{\varphi_1}{\longrightarrow} \mathcal{F}^0=\mathcal{O}_{\mathbb{C}^n,0} \overset{\pi}{\longrightarrow} \mathcal{O}_X$$ is $G$-equivariantly exact. This finishes the base case.

Assume inductively that for some $i\geq 1$, we have constructed free $\mathcal{O}_{\mathbb{C}^n,0}$-modules $\mathcal{F}^0, \ldots, \mathcal{F}^i$ equipped with $G$-action and $G$-equivariant differentials $\varphi_1,\ldots,\varphi_i$ such that the sequence  $$\mathcal{F}^i \overset{\varphi_i}{\longrightarrow}\mathcal{F}^{i-1} \overset{\varphi_{i-1}}{\longrightarrow}\cdots \overset{\varphi_{1}}{\longrightarrow} \mathcal{F}^0  \overset{\pi}{\longrightarrow} \mathcal{O}_{\mathbb{C}^n,0}/I_X\longrightarrow 0$$ is exact. Note that the $\mathcal{O}_{\mathbb{C}^n,0}$-module $\mathcal{K}_i=\ker \varphi_{i-1}$, equipped the natural $G$-action, is finitely generated. Set $V_i:=\mathcal{K}_i/\mathfrak{m}\mathcal{K}_i$ and $\mathcal{F}^{i+1}:=\mathcal{O}_{\mathbb{C}^n,0}^{\oplus \dim V_i}$. Repeating the arguments in the previous step, we get an $G$-action on $\mathcal{F}^{i+1}$ and a $G$-equivariant differential $$\varphi_{i+1}: \mathcal{F}^{i+1}\rightarrow \mathcal{F}^i$$ whose image is exactly $\ker \varphi_i$. This ends the induction.

Finally, since $\mathcal{O}_{\mathbb{C}^n,0}$ is a regular local ring of dimension $n$, the projective dimension of $\mathcal{O}_{\mathbb{C}^n,0}/I_X$ is also finite and $\leq n$. Hence, the process in the induction argument must stop after a finite number of times. This gives the desired $G$-equivariant resolution.
\end{proof}
Now, let  $(X,0) \hookrightarrow (\mathbb{C}^n,0)$ be an inclusion of germs of complex spaces where $(X,0)$. Suppose that $G$ be a Lie compact group acting holomorphically and linearly on $(\mathbb{C}^n,0)$ leaving $(X,0)$ invariant. We fix a $G$-equivariant resolution of the structure sheaf $\mathcal{O}_{X,0}$
\begin{equation} \label{gresfix}
\mathcal{F}^\ubul: 0 \longrightarrow \mathcal{F}^r \overset{\varphi_r}{\longrightarrow}\mathcal{F}^{r-1} \overset{\varphi_{r-1}}{\longrightarrow}\cdots \overset{\varphi_{1}}{\longrightarrow} \mathcal{F}^0  \overset{\pi}{\longrightarrow} \mathcal{O}_{X,0}= \mathcal{O}_{\mathbb{C}^n,0}/\mathcal{I}_X\longrightarrow 
\end{equation} which exists by Proposition \ref{Gres}.
\begin{lemma}[Relative Nakayama lemma] \label{relnaklem} Let $A$ be a local ring with maximal ideal $\mathfrak{m}_A$, $B$ a local $A$-algebra flat over $A$ and $M$ a finitely generated $B$-module. Let $x_1,\ldots,x_r \in M$ such that their images $\bar{x_1},\ldots,\bar{x_r}$ in the special fiber $M\otimes_A(A/\mathfrak{m}_A)$  generate $M/\mathfrak{m}_A M$ as a module over $B\otimes_A(A/\mathfrak{m}_A)$. Then $x_1,\ldots,x_r$ generate $M$ as a $B$-module. 

\end{lemma}
\begin{proof} Note that $\mathfrak{m}_A M =(\mathfrak{m}_A B)M$. By the flatness condition, 
$$M\otimes_A(A/\mathfrak{m}_A)\cong M/\mathfrak{m}_A M=M/(\mathfrak{m}_A B)M.$$
Let $N$ be the $B$-submodule generated by $x_1,\ldots,x_r$. By assumption, we can write $$M=N+(\mathfrak{m}_A B)M.$$ Since $\mathfrak{m}_A B \subset \mathfrak{m}_B$, we can apply the classical Nakayama lemma to conclude that $M=N$.
\end{proof}

\begin{prop} Let $(S,0)$ be a germ of complex spaces on which $G$ acts holomorphically. We equip $(\mathbb{C}^n,0)\times (S,0)$ with the diagonal $G$-action (note the $G$-action on $(\mathbb{C}^n,0)$ is always fixed to define the $G$-action on $(X,0)$). If there is a $G$-equivariant complex of free sheaves $\widetilde{\cF}^\ubul$ over $$(\mathbb{C}^n,0)\times (S,0)$$ extending (\ref{gresfix}) then $\widetilde{\cF}^\ubul$ is a free $G$-equivariant resolution of the structural sheaf of some deformation $\mathcal{X}\rightarrow (S,0)$ of $(X,0)$.
\end{prop}
\begin{proof} The only thing that needs to be proved is the exactness of $\widetilde{\cF}^\ubul$ which easily follows from the relative Nakayama lemma.
\end{proof}
Luckily, the converse is also true.
\begin{prop} \label{deftores} Let $\mathcal{X}\rightarrow (S,0)$ be a $G$-equivariant deformation of $(X,0)$ given in the product $(\mathbb{C}^n,0)\times (S,0)$ in view of Lemma \ref{gdefemb}. Then the  $G$-equivariant resolution $\mathcal{F}^{\ubul}\rightarrow \mathcal{O}_X$ (\ref{gresfix}) can be extended to a $G$-equivariant resolution $\widetilde{ \mathcal{F}}^{\ubul} \rightarrow \mathcal{O}_{\mathcal{X}}$ of the structure sheaf $\mathcal{O}_{\mathcal{X}}$ where each $\tilde{\mathcal F}^{i}$ is a free $\mathcal{O}_{(\mathbb{C}^n,0)\times (S,0)}$-module. In this case, we call $\widetilde{ \mathcal{F}}^{\ubul}$ a $G$-equivariant deformation of the $G$-equivariant free resolution $\cF^\ubul$ of the structre sheaf $\mathcal{O}_{X,0}$.
\end{prop}
\begin{proof} The proof will be a relative version of that of Proposition \ref{Gres}. Suppose that the germ $(X,0)\subset (\mathbb{C}^n,0)$ is defined by an $G$-stable ideal $I$ and the germ $\mathcal{X}\subset (\mathbb{C}^n \times S,0)$ is defined by a $G$-stable ideal $\mathcal{I}$. For simplicity, we introduce new notations: $A:=\mathcal{O}_{S,0}$, $B:=\mathcal{O}_{\mathbb{C}^n\times S,0}$ and $B_0:=\mathcal{O}_{\mathbb{C}^n,0}$, which are all local analytic algebras. In this way, $B_0 \cong B\otimes_A A/\mathfrak{m}_A$, $I\cong \mathcal{I}\otimes_A A/\mathfrak{m}_A$, $\mathcal{O}_{\mathcal{X}}=B/\mathcal{I}$ and $\mathcal{O}_{X,0}=B_0/I$, where $\mathfrak{m}_A$ is the unique maximal ideal of $A$. We aim to build free $B$-modules $\widetilde{\mathcal{F}}^i$ and maps $\tilde{\varphi}_i$ lifting $\mathcal{F}^i$ and $\varphi_i$ in (\ref{gresfix}), respectively. As before, the construction is done via induction.

\textbf{Step 1:  Base case.}
Let $K_0:=I \subset B_0$ and $V_0:=K_0/\mathfrak{m}_{B_0} K_0$ be the vector space of minimal generators of $I$ on the central fiber with a basis $v_1,\ldots,v_{m_0}$. In the proof of Proposition \ref{Gres}, we found holomorphic representatives $s_j\in I \subset B_0$ and averaged them with respect to a Haar measure on $G$ to get a $G$-equivariant finite-dimensionam subspace $W_0 \subset I$ lifting $V_0$. Now, we proceed in the same way but relatively.
Since $\mathcal{I}$ is flat over $A$, we have an exact sequence 
$$0 \rightarrow \mathfrak{m}_A\mathcal{I}\rightarrow \mathcal{I}\rightarrow I\rightarrow .0$$ Choose any representative $\widetilde{s}_j \in \mathcal{I}$ mapping to $s_j$ under the specialization $B\rightarrow B_0$. Now, form the averaged lifts $$\widetilde{\sigma}_j(z,s):=\int_G g\cdot (\widetilde{s}_j(g^{-1}\cdot (z,s)))d\mu(g)$$ where $g$ acts diagonally on $(z,s)\in \mathbb{C}^n\times S$ with the linear action on $\mathbb{C}^n$ and the given action on $S$. We can check that $\widetilde{\sigma}_j \in \mathcal{I}$ and that they are $G$-equivariant, specializing to $s_j$ on the central fiber. 

Let $\widetilde{W}_0 \subset \mathcal{I}$ be the $B$-submodule generated by $\tilde{\sigma}_1, \ldots,\widetilde{\sigma}_{m_0}$. Define $\widetilde{\mathcal{F}}_1:=B^{\oplus m_0}$ with the natural action induced from the diagonal $G$-action and the representation of $G$ on $V_0$ as in the proof of Proposition \ref{Gres}. With respect to the standard basis $\lbrace e_1,\ldots,e_{m_0}\rbrace$ of $\widetilde{\mathcal{F}}^1$, consider the map $$\widetilde{\varphi}_1: \widetilde{\mathcal{F}}^1 \rightarrow B, \; e_j \mapsto \tilde{\sigma}_j.$$ By construction, $\widetilde{\varphi}_1$ is $G$-equivariant, reducing to $\varphi_1$ on the central fiber. Finally, by Lemma \ref{relnaklem} applied to $M:=\mathcal{I}$, we have that $\mathrm{Im} \widetilde{\varphi}_1=\mathcal{I}$ after possibly shrinking the base $S$. Therefore, we get an sequence $$\widetilde{\mathcal{F}}^1 \overset{\varphi_1}{\longrightarrow} \widetilde{\mathcal{F}}^0=B \overset{\pi}{\longrightarrow} \mathcal{O}_{\mathcal{X}}$$ is $G$-equivariantly exact, extending the corresponding part of the resolution (\ref{gresfix}). This finishes the base case.

\textbf{Step 2: Induction step.}
Suppose that we constructed free $B$-modules $\widetilde{\mathcal{F}}^0,\ldots,\widetilde{\mathcal{F}}^i$ and $G$-equivariant maps $\widetilde{\varphi}_1,\ldots,\widetilde{\varphi}_i$ such that 
\[
\begin{tikzcd}[column sep=large, row sep=huge]
\widetilde \cF^i \ar[r,"\widetilde\varphi_i"] \ar[d,"\rho_i"] &
\widetilde \cF^{i-1} \ar[r,"\widetilde\varphi_{i-1}"] \ar[d,"\rho_{i-1}"] &
\cdots \ar[r] &
\widetilde \cF^1 \ar[r,"\widetilde\varphi_1"] \ar[d,"\rho_1"] &
\widetilde \cF^0 \ar[r,"\widetilde\pi"] \ar[d,"\rho_0"] &
\mathcal{O}_{\mathcal X,0} \ar[r] \ar[d,"\mathrm{id}"] & 0 \\
\cF^i \ar[r,"\varphi_i"] &
\cF^{i-1} \ar[r,"\varphi_{i-1}"] &
\cdots \ar[r] &
\cF^1 \ar[r,"\varphi_1"] &
\cF^0 \ar[r,"\pi"'] &
\mathcal{O}_{X,0} \ar[r] & 0
\end{tikzcd}
\]
where the rows are exact. Let $K_i:=\ker \varphi_i$ and $V_i:=K_i/\mathfrak{m}_{B_0}K_i$. Note that $K_i$ is $G$-stable and this $G$-action descends to a $G$-action on $V_i$. Set $\widetilde{\mathcal{F}}^{i+1}:=B^{\oplus m_{i+1}}$ where $m_{i+1}=\dim V_{i}$. Repeating the arguments in Step 1, we can construct a $G$-equivariant map $\widetilde{\varphi}_{i+1}: \widetilde{F}^{i+1} \rightarrow \widetilde{F}^{i}$ sitting in the commutative diagram
\[
\begin{tikzcd}[column sep=large, row sep=huge]
\widetilde{\mathcal{F}}^{i+1}\ar[r,"\widetilde\varphi_{i+1}"] \ar[d,"\rho_{i+1}"']&\widetilde \cF^i \ar[r,"\widetilde\varphi_i"] \ar[d,"\rho_i"'] &
\widetilde \cF^{i-1} \ar[r,"\widetilde\varphi_{i-1}"] \ar[d,"\rho_{i-1}"] &
\cdots \ar[r] &
\widetilde \cF^1 \ar[r,"\widetilde\varphi_1"] \ar[d,"\rho_1"] &
\widetilde \cF^0 \ar[r,"\widetilde\pi"] \ar[d,"\rho_0"] &
\mathcal{O}_{\mathcal X,0} \ar[r] \ar[d,"\mathrm{id}"] & 0 \\
\mathcal{F}^{i+1}\ar[r,"\varphi_{i+1}"]&\cF^i \ar[r,"\varphi_i"] &
\cF^{i-1} \ar[r,"\varphi_{i-1}"] &
\cdots \ar[r] &
\cF^1 \ar[r,"\varphi_1"] &
\cF^0 \ar[r,"\pi"'] &
\mathcal{O}_{X,0} \ar[r] & 0.
\end{tikzcd}
\] This completes the induction step.

\textbf{Step 3. Termination.} The process must terminate after a finite number of steps. This follows from the finiteness of the resolution (\ref{gresfix}) and the relative Nakayama lemma - Lemma \ref{relnaklem}.
\end{proof}
When the $G$-equivariant resolution (\ref{gresfix}) of the central fiber is fixed, the constructed $G$-equivariant resolution in Proposition \ref{deftores} might not be unique. Nevertheless, if one has two such $G$-equivariant resolutions $\widetilde{\cF}^\ubul_1 $ and  $\widetilde{\cF}^\ubul_2 $, then one can construct a $G$-equivariant map $\sigma$ from one to another
\[
\begin{tikzcd}[column sep=large, row sep=huge]
\widetilde{\mathcal{F}}^\ubul_1: \ar[d,"\sigma"']&\widetilde 0 \ar[r]  &
\widetilde \cF^{r}_1 \ar[r,"\widetilde\varphi_{r}^1"] \ar[d,"\sigma_{r}"'] &
\cdots \ar[r] &
\widetilde \cF^1_1 \ar[r,"\widetilde\varphi_1^1"] \ar[d,"\sigma_1"] &
\widetilde \cF^0_1 \ar[r,"\widetilde\pi^1"] \ar[d,"\sigma_0"] &
\mathcal{O}_{\mathcal X,0} \ar[r] \ar[d,"\mathrm{id}"'] & 0 \\
\widetilde{\mathcal{F}}^\ubul_2:&0 \ar[r] &
\widetilde{\cF}^{r}_2 \ar[r,"\widetilde{\varphi}_{r}^2"] &
\cdots \ar[r] &
\widetilde{\cF}^1_2 \ar[r,"\widetilde{\varphi}_1^2"] &
\widetilde{\cF}^0_2 \ar[r,"\widetilde{\pi}^1"] &
\mathcal{O}_{\mathcal{X},0} \ar[r] & 0.
\end{tikzcd}
\] reducing to the identity map on the central fiber. In particular, the matrices $\sigma_i$ are invertible in the germ $(\mathbb{C}^n,0)\times (S,0)$.  

\begin{proof}[Proof of Theorem \ref{gcordef}]
This is just a summary of what we have done so far.
\end{proof}
The following simple result will be useful in the sequel.
\begin{lemma}\label{uniextact} Let $G$ be a complex reductive Lie group and $K$ be its maximal compact subgroup whose complexification is exactly $G$. Let $(X,0) \hookrightarrow (\mathbb{C}^n,0)$ be an inclusion of germs of complex spaces. Suppose that $K$ acts holomorphically on $(\mathbb{C}^n,0)$ leaving $(X,0)$ invariant. There exists a unique (local) holomorphic $(G,K)$-action $(\mathbb{C}^n,0)$ leaving $(X,0)$ invariant and extending the given $K$-action.
\end{lemma}
\begin{proof}  The proof is similar to \cite[Corollary 5.1]{Doa21}. So, we shall omit the details.
\end{proof}

Finally, in \cite{Cat78}, the author studies deformations of pairs $(X,G)$ where $X$ is a complex compact space and $G$ is a complex Lie group acting holomorphically on $X$, i.e. a deformation of $(X,G)$ means a usual deformation  $\pi: \mathcal{X} \rightarrow (S,0) $ of $X$ such that there is an $G$-action on $\mathcal{X}$ extending the given $G$-action on $X$ and the $G$-action on $(S,0)$ is required to be trivial. In particular, the $G$-action on $\mathcal{X}$ preserves the fibers. With this setting, the main result of \cite{Cat78} is the existence of a semi-universal deformation for $(X,G)$. Inspired by this, we define an analogous type of deformations on the side of germs of complex spaces.

\begin{defn} \label{defdefpair}
 Let $G$ be a complex Lie group and $K$ be a compact subgroup of $G$ which determines a holomorphic $(G,K)$-action on a gems of complex spaces $(X,0)$. A real analytic (resp. holomorphic)  deformation of $((X,0),(G,K))$ is a usual deformation of $(X,0)$ $\pi$: $(\mathcal{X},0)\rightarrow (S,0)$ equipped with a real analytic (resp. holomorphic) $(G,K)$-action on $X$ extending the given (resp. holomorphic) $(G,K)$-action on $(X,0)$ and the real analytic (resp. holomorphic) $(G,K)$-action on $(S,0)$ is required to be trivial.
\end{defn}
\begin{defn} A deformation  $\pi:(\cX,0)\rightarrow (S,0)$ of the pair $((X,0),(G,K))$ is said to be \textit{semi-universal} if for any other  deformation  $\rho:(\cY,0)\rightarrow (T,0)$ of $((X,0),(G,K))$, there exists a  morphism of germs of complex spaces $\phi: (T,0) \rightarrow (S,0)$ whose differential is unique  such that $\rho:(\cY,0)\rightarrow (T,0)$ is $(G,K)$-equivariantly isomorphic to the pullback $(\cX,0)\times_{(S,0)} (T,0) \rightarrow (T,0)$ along the morphism $\phi$.
\end{defn}
\begin{rmk} \label{defequvsdefpai} Of course, any deformation in the sense of Definition \ref{defdefpair} is a deformation in the sense of Definition \ref{defdefequ}. Therefore, what we have done so far is still valid for deformations of $((X,0),(G,K))$. Conversely, given a deformation $\pi: (\mathcal{X},0) \rightarrow (S,0)$ in the sense of Definition \ref{defdefequ}, if one takes the germ $(S^{(G,K)},0)$ fixed by the $(G,K)$-action, pulling back the deformation $\pi:(\mathcal{X},0) \rightarrow (S,0)$ along the natural inclusion $(S^{(G,K)},0) \rightarrow (S,0)$ gives a deformation in the sense of Definition \ref{defdefpair}. 
\end{rmk}
\begin{rmk} In a similar fashion, for $G$ a compact group, one can define the space $T^1_{((X,0),G)}$ of \textit{first order deformations} and the space $T^2_{((X,0),G)}$ of \textit{obstructions to deformation extensions} for the pair $((X,0),G)$. They turn out be canonically isomorphic to $(T^1_{(X,0)})^G$ and $(T^2_{(X,0)})^G$, respectively. This point shall be also clear from our construction in the sequel.
\end{rmk}
\begin{defn} \label{defdefpairrig}
 Let $G$ be a complex Lie group and $K$ be a compact subgroup of $G$ which determines a holomorphic $(G,K)$-action on a germ of complex spaces $(X,0)$. The pair  $((X,0),(G,K))$ is said to be \textit{rigid} if  every deformation of $((X,0),(G,K))$ over a germ of complex spaces $(B,0)$ is trivial, i.e. isomorphic to $(X,0)\times (B,0) \rightarrow (B,0)$. 
\end{defn}
\begin{rmk} Obviously, the rigidity of the germ $(X,0)$ implies that of the pair $((X,0),(G,K))$ but the converse is not true. Several examples are given at the end of  Section \ref{sec: exi}.
\end{rmk}
\section{Equivariant first order deformations of germs of complex spaces} \label{sec: doupoi} In this section, we recall Donin's study of deformations over double points (cf. \cite[Section 2]{Don72}), into which we shall integrate the group actions along the way. 

Assume that the $G$-equivariant free resolution (\ref{gresfix}) of the structure sheaf $\mathcal{O}_{X,0}$ is defined over some small open neighborhood $\Omega$ of $0$ in $\mathbb{C}^n$, i.e. the holomorphic matrices $\varphi_1,\ldots,\varphi_r$ are defined over $\Omega$. From now on, all balls or polydiscs under consideration are assumed to lie stricly in $\Omega$. 

Let $K$ be a small ball centered at $0$ in $\mathbb{C}^n$. Denote by $\mathcal{F}^1(\tilde K)$ the space of $n$-tuples $(q_1,\cdots,q_r)$ where each $q_i$ is a matrix with coefficients holomorphic in $\mathring{K}$ and continuous on $\bar{\tilde K}$ such that $q_i: \mathcal{F}^i\rightarrow \mathcal{F}^{i-1}$ is a morphism of sheaves. Then $\mathcal{F}^1(\tilde K)$ is a Banach  vector space with the obvious norm. In the same way, denote by $\mathcal{F}^2(\tilde K)$ the Banach vector space of $n$-tuples $(p_2,\cdots,p_r)$ where each $p_i$ is a matrix with coefficients holomorphic in $\mathring{\tilde K}$ and continuous on $\bar{\tilde K}$ such that $q_i: \mathcal{F}^i\rightarrow \mathcal{F}^{i-2}$ is a morphism of sheaves. Consider the map
$$\delta_{\tilde K}^1:\mathcal{F}^1(\tilde K) \rightarrow \mathcal{F}^2(\tilde K), \;(q_1,\cdots,q_r) \mapsto (q_{2}\circ q_1,\ldots, q_{i-1}q_i,\ldots, q_{r_1}\circ q_r). $$ As shown in \cite[Section 2.1]{Don72}, $\delta_K^1$ is an analytic map so that $Z_K:=(\delta_K^1)^{-1}(0)$ is an  analytic Banach space. Finally, we denote by $\mathcal{F}^0(K)$ the analytic Banach manifold of $(r+n)$-tuples $(\gamma_0,\ldots,\gamma_r,z_1+\Phi_1,\ldots, z_n+\Phi_n)$ where $\gamma_i$ is an invertible matrix function holomorphic on the interior of $K$ and continuous on its boundary giving rise to an isomorphism of sheaves $\gamma_i: \mathcal{F}^i \rightarrow \mathcal{F}^i$ for each $i$ and $\Phi_j$ is a holomorphic function on $K$, continuous on its boundary for each $j$ such that $\Phi:=(z_1+\Phi_1,\ldots,z_n+\Phi_n): K \rightarrow K$ is a small holomorphic coordinate change. In $\mathcal{F}^0(K)$, we have a distinguished point $(\mathrm{id},z):=(\mathrm{id},\ldots,\mathrm{id},z_1,\ldots,z_n)$ where $z_1,\ldots,z_n$ are the coordinate functions on $\mathbb{C}^n$. Similarly we have $\varphi:=(\varphi_1,\ldots,\varphi_r)$ (coming from the fixed resolution (\ref{gresfix})) as a distinguished point of $\mathcal{F}^1(K)$ and $0:=(0,\ldots,0)$ as a distinguished point of $\mathcal{F}^{2}(K)$. 

From now on, we always work over sufficiently small neighborhoods of the distinguished points. Moreover, if $K' \subset K$ is another small neighborhood of $0 \in \mathbb{C}^n$, then we have a natural restriction map $\mathrm{res}: \mathcal{F}^i(K)\rightarrow \mathcal{F}^i(K')$. Define a map $\omega: \mathcal{F}^0(K)\times \mathcal{F}^1(K)\rightarrow \mathcal{F}^1(K')$ given by
$$(\gamma_0,\ldots,\gamma_r,z_1+\Phi_1,\ldots,z_n+\Phi_n) \times (q_1,\ldots,q_k)\mapsto (q_1',\ldots,q_k')$$ where for $z\in K$ $q_i'(z)=\gamma_{i-1}(z)q_i(\Phi^{-1}(z))\gamma_i^{-1}(z)$ if we denote $\Phi:=(z_1+\Phi_1,\ldots,z_n+\Phi_n)$.  In this way, $\omega$ can be regarded as an action of $\mathcal{F}^0(K)$ on $\mathcal{F}^1(K)$, leaving $Z_K$ invariant. In particular, $\omega$ induces a map $\delta_K^0: \mathcal{F}^0(K) \rightarrow \mathcal{F}^1(K)$ given by the action of $\mathcal{F}^0(K)$ on the distinguished point $\varphi:=(\varphi_1,\ldots,\varphi_r)$ of $\mathcal{F}^1(K)$. \cite[Section 2.2]{Don72} indicates that given $(S,0)$ a germ of complex spaces, a deformation of $(X,0)$ over $(S,0)$ is equivalent to a choice of a sufficiently small neighborhood $K$ together with an analytic morphism $(S,0) \rightarrow (Z_K,\varphi)$.

Now, we take the group actions into account. Since $G$ is compact the small ball $K$ can be chosen to be further $G$-stable. Note that on $\mathcal{F}^i$ $(i=1,\ldots,r)$, we already have a $G$-action. For a matrix of holomorphic function on $K$ and continuous on its boundary $P:\mathcal{F}^j \rightarrow \mathcal{F}^{j-1}$, we define the action of $g \in G$ on $P$ by the formula
\be \label{gactonf}(g\cdot P)(z):=g\circ P(g^{-1}\cdot z)\circ g^{-1}.\ee  This in turn induces natural \textit{diagonal} well-defined $G$-actions on $\mathcal{F}^i(K)$ ($i=0,1,2$). 
\begin{lemma} \label{lem:fixandequi} With the above settings, \begin{itemize}
 \item[1.] the distinguished points in $\mathcal{F}^0(K)$, $\mathcal{F}^{1}(K)$ and $\mathcal{F}^2(K)$ are fixed by the corresponding $G$-actions;

\item[2.] the map $\delta^i: \mathcal{F}^i(K) \rightarrow \mathcal{F}^{i+1}(K)$ is $G$-equivariant ($i=0,1$). 
\end{itemize}
\end{lemma}
\begin{proof} For the first statement, the only thing which is not so clear is the distinguished point $\varphi:=(\varphi_1,\ldots,\varphi_r) \in \mathcal{F}^1(K)$  being fixed by the $G$-action. However, this follows directly from the fact that the free resolution (\ref{gresfix}) is $G$-equivariant.

Now, it remains to show the second statement. For $g\in G$ we compute:
\begin{align*}
(\delta^1(g\cdot q))_i(z)
  &= (g\cdot q)_{i-1}(z)\;\cdot\; (g\cdot q)_i(z) \\
  &= \big( g\, q_{i-1}(g^{-1}z)\, g^{-1} \big)
     \big( g\, q_i(g^{-1}z)\, g^{-1} \big) \\
  &= g\, \big( q_{i-1}(g^{-1}z)\, q_i(g^{-1}z) \big)\, g^{-1} \\
  &= \big( g\cdot (\delta^1(q)) \big)_i(z).
\end{align*}
Thus 
\[
\delta^1(g\cdot q) = g\cdot \delta^1(q)
\] giving the $G$-equivariance of $\delta^1$.

Next, let $g\in G$ and $\eta:=(\gamma,\Phi):=(\gamma_0,\ldots, \gamma_r,z_1+\Phi_1,\ldots,z_n+\Phi_n)\in \mathcal{F}^0(K)$. Denote $g\cdot\eta = (\widetilde\gamma,\widetilde\Phi)$. 
Then
\[
\widetilde\gamma_i(z) = g\, \gamma_i(g^{-1}z)\, g^{-1},
\qquad
\widetilde\Phi = g\circ \Phi\circ g^{-1}.
\] We compute:
\begin{align*}
(\delta^0(g\cdot\eta))_i(z)
  &= \widetilde\gamma_{i-1}(z)\; (p_i\circ\widetilde\Phi^{-1})(z)\; 
     \widetilde\gamma_i(z)^{-1} \\
  &= \big( g\,\gamma_{i-1}(g^{-1}z)\, g^{-1} \big)\;
     p_i(\widetilde\Phi^{-1}(z))\;
     \big( g\,\gamma_i(g^{-1}z)^{-1}\, g^{-1} \big).
\end{align*}
Since $\widetilde\Phi^{-1} = g\circ \Phi^{-1}\circ g^{-1}$, we have
\[
p_i(\widetilde\Phi^{-1}(z))
  \;=\; p_i\!\big( g(\Phi^{-1}(g^{-1}z)) \big).
\]
Because the central resolution is $G$--equivariant, we have
\[
p_i(g w) \;=\; g\, p_i(w)\, g^{-1},
\qquad \forall\, w,
\]
hence
\[
p_i\!\big( g(\Phi^{-1}(g^{-1}z)) \big)
  \;=\; g\, p_i(\Phi^{-1}(g^{-1}z))\, g^{-1}.
\] Substituting:
\begin{align*}
(\delta^0(g\cdot\eta))_i(z)
  &= g\,\gamma_{i-1}(g^{-1}z)\, g^{-1}\;
     \big( g\, p_i(\Phi^{-1}(g^{-1}z))\, g^{-1} \big)\;
     g\, \gamma_i(g^{-1}z)^{-1}\, g^{-1} \\
  &= g\;\big(
        \gamma_{i-1}(g^{-1}z)\,
        p_i(\Phi^{-1}(g^{-1}z))\,
        \gamma_i(g^{-1}z)^{-1}
      \big)\; g^{-1} \\
  &= \big( g\cdot (\delta^0(\eta)) \big)_i(z).
\end{align*}
Thus
\[
\delta^0(g\cdot\eta) = g\cdot \delta^0(\eta)
\] showing the $G$-equivariance of $\delta^1$.
\end{proof}
In this way, the Banach analytic space $Z_K=(\delta_K^1)^{-1}(0)$ is also $G$-stable. By Proposition \ref{deftores} and Theorem \ref{gcordef}, we have the following $G$-equivariant of Donin's result.

\begin{theorem} \label{defvartodefres} Given $(S,0)$ a germ of complex spaces on which a compact complex Lie group $G$ acts holomorphically, a $G$-equivariant deformation of $(X,0)$ over $(S,0)$ is equivalent to a choice of a sufficiently small neighborhood $K$ together with an $G$-equivariant analytic morphism $(S,0) \rightarrow (Z_K,\varphi)$.

\end{theorem}

\begin{cor} \label{defvartodefrespai} Given $(S,0)$ a germ of complex spaces, a  deformation of $((X,0),G)$ over $(S,0)$  is equivalent to a choice of a sufficiently small neighborhood $K$ together with an analytic morphism $(S,0) \rightarrow (Z_K^G,\varphi)$ where $Z_K^G$ is the set of fixed points of $Z_K$ by the $G$-action.
\end{cor}
\begin{proof}
This follows directly from Definition \ref{defdefpair} and Remark \ref{defequvsdefpai}. 
\end{proof}
\section{Existence of equivariant semi-universal deformations} \label{sec: exi}With the same settings and notations as in the two preceding sections, in this section, we shall construct a $G$-equivariant semi-universal deformation of the germ of complex spaces $(X,0)$ by integrating the $G$-action into Donin's original construction (cf. \cite[Section 3]{Don72} for full details and \cite[Section 3]{Pue82} for a glimpse). As before, since $G$ is compact one can assume the $G$-action on $(\mathbb{C}^n,0) \supset (X,0) $ is linear after a change of holomorphic coordinates (see \cite{Kau65} or \cite[Section 2.2]{Akh95})

We choose a sufficiently small $G$-stable ball $K$ and a privileged polydisc $\wK \subset K$. Denote $\cF^i(\wK):=\mathrm{Res}_{K \rightarrow \wK}(\cF^i(K))$ where $\mathrm{Res}_{K \rightarrow \wK}$ is the natural restriction map. By abuse of notations, we still use $\delta^i$ to denote the induced morphism. It is routine to check that $\cF^i(\wK)$ is still a Banach vector space ($i=0,1,2$) and the map $\delta^i$ is still analytic ($i=0,1$). We equipped $\cF^i(\wK)$ with an induced $G$-action from the one on $\cF^i(K)$ as follows. For $P \in \cF^i(\wK)$ and $g \in G$, choose any extension $\widetilde{P}\in \cF^i(K)$ with $\widetilde{P}\mid_{\wK} =P$ and define 
\be\label{indgactonf} g\cdot P:=(g\cdot \widetilde{P})\mid_{\wK}.
\ee
By analytic continuation, this gives rise to well-defined $G$-actions with respect to which the morphism $\delta^i$ is $G$-equivariant. As a sequence, $Z_{\wK}:=(\delta^1)^{-1}(0)$ is still a $G$-stable Banach analytic space. Note that Lemma \ref{lem:fixandequi} is obviously still valid for $\cF(\wK)^i$. So,
taking the derivation with respect to the distinguished points gives  $G$-equivariant maps of tangent spaces 
\be \label{gtangent} T_{(\mathrm{id,z})}\cF^0(\wK) \overset{d\delta^0}{\longrightarrow}T_{\varphi}\cF^1(\wK)\overset{d\delta^1}{\longrightarrow} T_{0}\cF^2(\wK)
\ee 
\begin{lemma} Let $\mathbf{0}$ be one of the three distinguished points. Then there exists a canonical $G$-equivariant isomorphism 
$$T_{\mathbf{0}}\cF^i(\wK) \overset{\cong}{\rightarrow}\cF^i(\wK)$$

\end{lemma}
\begin{proof}
Since $\cF^i(\wK)$ is a Banach vector space, it is canonically identified with its tangent space $T_{\mathbf{0}}\cF^i(\wK)$. It remains to check if the $G$-action on $\cF^i(\wK)$ and the induced $G$-action on $T_{\mathbf{0}}\cF^i(\wK)$ are compatible. However, it follows from the fact the $G$-actions (\ref{gactonf}) and (\ref{indgactonf}) are linear (they are defined by taking pullback and conjugation) so their derivative at the fixed point gives the same formula.
\end{proof}
In this way, the sequence (\ref{gtangent}) can be rewritten simply as \be \label{gtangenttof} \cF^0(\wK) \overset{d\delta^0}{\longrightarrow}\cF^1(\wK)\overset{d\delta^1}{\longrightarrow} \cF^2(\wK).
\ee  which allows to characterize deformations of $(X,0)$ or equivalently deformations of the fixed resolution (\ref{gresfix}) over the double point. In other words, passing to the inductive limit over a system of neighborhoods $K$ of $0$, one gets a complex of sheaves
\begin{equation*} \bar{\cF}^0 \overset{d\delta^0}{\longrightarrow}\bar{\cF}^1\overset{d\delta^1}{\longrightarrow} \bar{\cF}^2.
\end{equation*} 
\begin{thrm}$\mathrm{(}$\cite[Proposition 2.1]{Don72}$\mathrm{)}$
$\mathcal{R}:=\ker(d\delta^0)/\mathrm{im}(d\delta^0) $ is a sheaf over $X$ and one has an isomorphism  $$\mathcal{R}_0 \cong T^1_{(X,0)}$$ of $\mathbb{C}$-vector spaces.
\end{thrm}
\begin{prop} \label{ghyper} There exist
\begin{itemize}
\item[1.] a $G$-stable linear hyperplane $H$ in $\mathcal{F}^2(K)$ passing through the distinguished point $0$ such that $T_0 H\cong \mathrm{Im }\; d\delta^1$ and
\item[2.] a $G$-equivariant linear projection $\pi: \cF(\wK)^2 \rightarrow H$ such that $\mu^{-1}(0)$ is a $G$-stable Banach analytic submanifold of $\cF(\wK)$ with $T_0\mu^{-1}(0)=\ker d\delta^1=T_{\varphi}Z_{\wK}$ where $\mu:=\pi \circ \delta^1$.
\end{itemize} 
\end{prop}
\begin{proof}
Since $\wK$ is priviledged and $d\delta^1$ is $G$-equivariant, $\mathrm{Im }\; d\delta^1$ is a $G$-stable closed complemented subspace in  $\cF^2(\wK)$. Hence,  we can take $H$ to be any $G$-stable linear hyperplane in  $\cF^2(\wK)$, whose tangent space at the distinguished point is $G$-equivariantly identified with $\mathrm{Im}\;d\delta^1$ (or simply take $H=\mathrm{Im}\;d\delta^1$). As $G$ is compact the existence of $\pi$ follows from the fact that we can take a linear projection on $H$ and make it $G$-equivariant by taking the average with respect to a Haar measure on $G$. Next, by construction, the differential of $\mu=\pi\circ \delta^1$ at the distinguished point is surjective. Hence, $\mu^{-1}(0)$ is a Banach analytic submanifold of $\mathcal{F}^1(K)$. The $G$-stability  of $\mu^{-1}(0)$ follows from the $G$-equivariance of $\mu$ being the composition of two $G$-equivariant maps.  Finally, it is routine to check the last chain of equivalences.
\end{proof}

\begin{proof}[Proof of Theorem \ref{thm: compact}]
Let $N$ be a  $G$-stable hyperplane in $H$ passing through the distinguished point $0$ such that we have a $G$-decomposition of tangent spaces
$$T_0N\oplus \mathrm{Im }\;d\delta^0 =T_0 H$$ where $H$ is given by Proposition \ref{ghyper}. In this way $T_0N$ is of dimension $l$ so that $N$ is a finite-dimensional manifold of dimension $l$ and $S:=N\cap Z_{\widetilde K} $ is a finite-dimensional analytic subspace in $Z_{\widetilde K}$, whose tangent space at the distinguished point is of dimension $l$. Moreover, $S$  inherits a natural $G$-action (since $Z_{\widetilde K}$ and $N$ do) in a way that the natural inclusion $\iota: S \hookrightarrow Z_{\widetilde K}$ is $G$-equivariant. For the sake of Theorem \ref{defvartodefres} and Theorem \ref{gcordef}, this gives a rise to a $G$-equivariant deformation  $\pi: (\mathcal{X},0) \rightarrow (S,0)$ of $(X,0)$. 

Finally, we turn to the equivariant semi-universality. Note also that $\ker d\sigma^0$ is also a $G$-stable closed completemented subspace in $\mathcal{F}^0(\tilde K)$. Therefore, for the compactness of $G$, we can take a $G$-stable hyperplane $P$ in $\mathcal{F}^0(\tilde K)$ whose tangent space at the distinguished point is a complementary subspace to $\ker d\delta^0$. Recall that the equivariant semi-universaltity means that any other $G$-equivariant deformation of $(X,0)$ is defined by the pullback of the family constructed along a $G$-equivariant morphism of germs of complex spaces.  It can be shown by following the same lines as in the last two paragraphs of the proof of \cite[Theorem 3.1]{Don72} using crucially the fact that the natural morphism $P\times S \rightarrow Z_{\tilde K}$ is an isomorphism of Banach analytic spaces, which is further naturally $G$-equivariant by the choice of $P$ and $S$. As a matter of fact, any map appearing in the argument is $G$-equivariant. In particular, the $G$-equivariant semi-universal deformation we have just constructed is unique up to $G$-equivariant isomorphisms.
\end{proof}
\begin{proof}[Proof of Theorem \ref{thm: reductive}]
By definition, the $\widetilde G$-action is global. Thus, repeating the proof of Theorem \ref{thm: compact}, we obtain a $\widetilde G$-equivariant analytic maps of pointed (Banach) analytic spaces $$\iota: (S,0) \hookrightarrow (Z_{\widetilde K},0).$$
It follows now from a germ version of \cite[Corollary 5.1]{Doa21} that we can equip local holomorphic $(G,\widetilde G)$-actions on $(S,0)$ and $(Z_{\widetilde K},0)$ extending the initial $\widetilde G$-actions in a way that $\iota$ is $(G,\widetilde G)$-equivariant. By Donin's original construction of $Z_{\widetilde K}$, it corresponds to a semi-universal deformation of $\cF^\ubul$ (see (\ref{gresfix}) without taking into account the group actions). However, this time $\cF^\ubul$ carries local holomorphic $(G,\widetilde G)$-actions extending the initial $\widetilde G$-actions. Hence, it gives rise to a local holomorphic $(G,\widetilde G)$-equivariant deformation $\pi: (\mathcal{X},0) \rightarrow (S,0)$ of $(X,0)$. The fact that the restriction of the local $G$-action on the total space $(\mathcal{X},0)$ to the central fiber is nothing but the $(G,\widetilde G)$-action on $(X,0)$ follows from Lemma \ref{uniextact}. This finishes the proof.
\end{proof}
At last, we turn to deformations in the sense of Definition \ref{defdefpair}. 
\begin{proof}[Proof of Theorem \ref{thm: compactpai} and Theorem \ref{thm: reductivepai}] It follows directly from Corollary \ref{defvartodefrespai} coupled with Theorem \ref{thm: compact}, Theorem \ref{thm: reductive} and their proofs.
\end{proof}  
\begin{proof}[Proof of Corollary \ref{thm: reductivepairig}] Since the base $(S^{(G,\tilde{G})})$ of the semi-universal deformation $\pi: (\mathcal{X},0) \rightarrow (S^{(G,\tilde G)},0)$ of $((X,0),(G,\tilde G))$ is obtained by taking the fixed points of $(S,0)$, $$\dim_{\mathbb{C}}T_0S^{(G,\tilde G)} \leq \dim_{\mathbb{C}}T_0 S^{\tilde{G}}= \dim_{\mathbb{C}} \left(T^1_{(X,0)}\right)^{\tilde{G}}=0.$$ Hence, the semi-universal deformation is $(X,0) \rightarrow  \lbrace 0\rbrace$ so that the rigidity follows from the semi-universality.
\end{proof}
\begin{proof}[Proof of Corollary \ref{thm: compactpaipro}] By construction, one has a natural inclusion $(S^G,0)\rightarrow (S,0)$ of germs of complex spaces. Moreover, $T_0S^G \cong T_0S$. Therefore, the first statement follows from the smoothness of $(S^G,0)$ and the implicit function theorem. The second statement now is a direct consequence of the semi-universality.
\end{proof}
\begin{proof}[Proof of Theorem \ref{thm: mil}] One can follows the same lines as in the proof of \cite[Theorem]{Wal80} with an observation that $G$ being a compact group is still semi-simple (existence of Haar measures + Weyl's unitary trick) so that all extensions under consideration therein still split.
\end{proof}
\begin{ex} \label{ex: unicor}Consider the germ of complex spaces $(X,0)$ where $$X:=\lbrace (x,y)\in \mathbb{C}^2 \mid xy=0\rbrace$$ and $0$ is the origin. A semi-universal deformation of $(X,0)$ is given by $\pi: (\cX,0) \rightarrow (\bC,0)$
 where \be \label{semiuni}\cX:=\lbrace (x,y,s)\in \mathbb{C}^3 \mid xy=s\rbrace \ee  and $\pi$ is the natural projection on the last factor. The reductive complex Lie group $G:=\bC^* \times \bC^*$ acts on $X$ by the rule 
\begin{align*}
G\times X &\rightarrow X\\
((\lambda,\mu),(x,y)) &\mapsto (\lambda x, \mu y)
\end{align*} 
which extends to a $G$-action on $\cX$ by the rule \begin{align*}
G\times \cX &\rightarrow \cX\\
((\lambda,\mu),(x,y,s)) &\mapsto (\lambda x, \mu y,\lambda \mu s).
\end{align*} If one defines a $G$-action on $\bC$ by the rule \begin{align*}
G\times \bC &\rightarrow \bC\\
((\lambda,\mu),s) &\mapsto \lambda \mu s
\end{align*} then $\pi$ is actually $G$-equivariant with respect to these actions. Of course, these $G$-actions are only local on the level of germs of complex spaces. Note also that since $0$ is the only point fixed by the $G$-action then the pair $((X,0),G)$ is rigid in the sense of Definition \ref{defdefpairrig}.
\end{ex}  
\begin{ex} Consider the germ of complex spaces $(X,0)$ where $$X:=\lbrace x:=(x_1,x_2,x_3)\in \mathbb{C}^3 \mid f_1(x):=x_1^2 +x_2^3=0 \wedge  f_2(x):=x_2^3+x_3^2=0\rbrace$$ is a complete intersection with $0$ as an isolated singularity. Let $s:=(s_1,\ldots,s_9)$ be coordinates of $\bC^9$. The maps
$$F_1(x,s_1,\ldots,s_7):= f_1(x)+s_1x_2+s_2x_3+s_3x_2x_3z+s_4x_2^2$$ and $$F_2(x,s_1,\ldots,s_7):= f_2(x)+s_5x_1+s_6x_2+s_7x_1x_2$$
determines a semi-universal deformation  of $(X,0)$ 
\begin{align*}
\pi:(\bC^3 \times \bC^7,0) &\rightarrow (\bC^9,0)\\
(x,s_1,\ldots,s_7) &\mapsto (s_1,\ldots,s_7,F_1(x,s_1,\ldots,s_7), F_2(x,s_1,\ldots,s_7))
\end{align*} (cf. \cite[Example II.1.18.1.(2)]{GLS25}). The reductive complex Lie group $G:=\bC^*$ acts on $\bC^3$ by the rule 
\begin{align*}
G\times \bC^3 &\rightarrow \bC^3\\
(\lambda,x) &\mapsto (\lambda^3 x_1, \lambda^2 x_2,\lambda^3 x_3).
\end{align*} which clearly leaves $X$ invariant. On $\bC^7$ and $\bC^9$, we define the $G$-actions
\begin{align*}
G\times \bC^7 &\rightarrow \bC^7\\
(\lambda,s_1,\dots,s_7) &\mapsto (\lambda^4 s_1, \lambda^3 s_2,\lambda s_3,\lambda^2 s_4,\lambda^3 s_5,\lambda^4s_6,\lambda s_7)
\end{align*} and 
\begin{align*}
G\times \bC^9 &\rightarrow \bC^9\\
(\lambda,s_1,\dots,s_9) &\mapsto (\lambda^4 s_1, \lambda^3 s_2,\lambda s_3,\lambda^2 s_4,\lambda^3 s_5,\lambda^4s_6,\lambda s_7,\lambda^6s_8, \lambda^6s_9)
\end{align*}
In addition, we equip the product $\bC^3 \times \bC^7$ with the natural diagonal $G$-actions. Finally, it is straightforward to check that  $\pi$ is $G$-equivariant with respect to these $G$-actions. Once again, the pair $((X,0),G)$ is rigid in the sense of Definition \ref{defdefpairrig}.
\end{ex}  
\begin{ex} Consider the $A_{2m-1}$-singularity $X:=\lbrace (x,y) \in \mathbb{C}^2|x^{2m}+y^2=0 \rbrace$ with a cylic action of $G:=\mathbb{Z}/m\mathbb{Z}$ given by $\zeta\cdot(x,y)=(\zeta x,y)$ for $\zeta$ a generator of $G$. The semi-universal deformation of $(X,0)$ is given by
$$\mathcal{X}:=\lbrace (x,y,s_0,\ldots,s_{2m-2})\in \mathbb{C}^2 \times \mathbb{C}^{2m-1} | x^{2m}+y^2+\sum_{i=0}^{2m-2}s_ix^i \rbrace$$ which can be equipped with a natural $G$-action
$$\zeta \cdot(x,y,s_0,\ldots,s_{2m-2}) =(\zeta x,y,\zeta^0 s_0,\ldots,\zeta^{-(2m-2)}s_{2m-2})$$ extending the given one. Together with the projection  $ \pi: (\mathcal{X},0)\rightarrow (S,0):=(\mathbb{C}^{2m-1},0)$, it gives the $G$-equivariant semi-universal deformation of $(X,0)$. However, $(S^G,0)=(\mathbb{C}^2,0)$ (only $s_0$ and $s_m$ are fixed by the extended $G$-action). So the total space of the semi-universal deformation of the pair $((X,0),G)$ is given by $$\lbrace (x,y,s_0,s_{m})\in \mathbb{C}^2 \times \mathbb{C}^{2m-1} | x^{2m}+y^2+s_0+s_m x^m \rbrace.$$ Of course, this can also be seen in terms of the Tjurina algebra $$T^1_{(X,0)}=\mathbb{C}\{ x\}/(x^{2m-1})=\mathrm{Span} \lbrace 1,\ldots,x^{2m-2} \rbrace$$ with the same induced $G$-action  $\zeta \cdot x =\zeta x$. In this way, $$T^1_{((X,0),G))}=(T^1_{(X,0)})^G=\mathrm{Span}\lbrace 1,x^m \rbrace.$$ Thus, symmetry rigidifies significantly the classification problem.
\end{ex}
\section{Optimality of the reductivity assumption} \label{sec: op} This section is devoted to showing that our reductivity assumption in what we have done so far is really optimal. 

Let $G=(\bC^* \times \bC^*) \rtimes (\bC \times \bC)$ which is obviously non-reductive. Consider once again the germ of complex spaces $(X,0)$ given
Example \ref{ex: unicor} but this time equipped with the following local holomorphic $G$-action \begin{align*}
(G,\mathrm{id})\times (X,0) &\rightarrow (X,0)\\
((\lambda,\mu,a,b),(x,y)) &\mapsto \left(\frac{\lambda x}{1-a x},\frac{\mu y}{1-b y}\right)
\end{align*} which generates four vector fields $$v_1:=x\partial_x,v_2:=y\partial_y,v_3:=x^2\partial_x,
v_4:=y^2\partial_y$$ on $(X,0)$ together with the following Lie bracket relations  \be \label{liebracen}\left\{\begin{matrix}
[v_1,v_2]=0;[v_1,v_3]=v_3; [v_1,v_4]=0\\
[v_2,v_3]=0;[v_2,v_4]=v_4;[v_3,v_4]=0
\end{matrix}\right.\ee
Of course, these vector fields constitute the Lie algebra $\mathfrak{g}$ associated to $G$. The aim is to show that this local $\mathbb{C}$-action on $(X,0)$ can not be extended to its semi-universal deformation (\ref{semiuni}). The method is the same as in  \cite{Doa20, Doa24a}: Lie-theoretic obstructions to extensions of vector fields.

First, we are going to describe explicitly the general form of vector fields on $\mathcal{X}$. Basically, $\mathcal{X}$ is obtained by gluing $U_x:\lbrace x\neq 0 \rbrace \subset \mathcal{X}$ with coordinates $(x,s)$  and $U_y:\lbrace y\neq 0 \rbrace \subset \mathcal{X}$ with coordinates $(y,t)$ along the overlap $U_{xy}:=U_x\cap U_y= \lbrace x\neq 0, y\neq 0 \rbrace \subset \mathcal{X}$, subject to the rule \begin{equation} \label{glurul} t=s, \; x=t/y.\end{equation} 
A general vector field on $U_x$ is of the form
\begin{equation*} 
v_x:=A(x,s)\partial_x +B(x,s)\partial_s
\end{equation*} where $A,B$ are elements of $\mathcal{O}_{U_x,0}$ and a general vector field on $U_y$ is of the form
\begin{equation*} 
v_y:=C(y,t)\partial_y +D(y,t)\partial_t
\end{equation*} where $C,D$ are elements of $\mathcal{O}_{U_y,0}$. Therefore, $v_x$ and $v_y$ constitute a vector field on $\mathcal{X}$ if they are the same on the overlap $U_{xy}$, i.e. 
\begin{equation} \label{loctoglo}
A(x,s)\partial_x +B(x,s)\partial_s=C(y,t)\partial_y +D(y,t)\partial_t.
\end{equation}
\begin{lemma} \label{vecgenfor} A vector field $v$ on $(\mathcal{X},0)$ whose restriction on $U_x$ is \begin{equation}  
v_x:=A(x,s)\partial_x +B(x,s)\partial_s
\end{equation} and on $U_y$ is 
\begin{equation}  
v_y:=C(y,s)\partial_y +D(y,s)\partial_s
\end{equation}
  must satisfy the following constraints
\be \label{const}\left\{\begin{aligned}
A(x,s)
&= x^2f(s)+xg(s)+h(s)\\[4pt]
B(x,s)
&= b(s)\\[4pt]
C(y,t)
&= -y^2\frac{h(t)}{t}-y\left(g(t)-\frac{b(t)}{t}\right)-tf(t) \\[4pt]
D(y,t)
&= b(t)
\end{aligned}
\right.
\ee  where $f,g,h$ and $b$ are elements of $\cO_{\bC,0}$ such that \be \label{constcoe} h(0)=b(0)=0.\ee 
\end{lemma}
\begin{proof}
By the gluing rule (\ref{glurul}), we have 
\be \label{vectrajac}
\left\{\begin{aligned}
\partial_y
&= -\frac{x^2}{s}\partial_x \\[4pt]
\partial_t
&= \frac{x}{s}\partial_x+\partial_s
\end{aligned}
\right.
\ee whose subsitution into (\ref{loctoglo}) implies that \[
\left\{\begin{aligned}
A(x,s)
&= -\frac{x^2}{s}C\left( \frac{s}{x},s\right)+\frac{x}{s}D\left( \frac{s}{x},s\right)\\[4pt]
B(x,s)
&= D\left( \frac{s}{x},s\right).
\end{aligned}
\right.
\]Therefore, (\ref{const}) and (\ref{constcoe}) follows from the holomorphicity.  
\end{proof}
 Suppose that the local $G$-action on $(X,0)$ can be extended to local $\bC$-actions on $(\mathcal{X},0)$ and on $(\mathbb{C},0)$, with respect to which the morphism $\pi$ is equivariant.  On one hand, for the sake of Lemma \ref{vecgenfor}, it gives rise to four vector fields on $\mathcal{X}$ of the form
$$\widetilde v_i:=(x^2f_i(s)+xg_i(s)+h_i(s))\partial_x+b_i(s)\partial_s$$ where $f_i,g_i,h_i$ are elements of $\mathcal{O}_{\bC,0}$, subject to the  condition \be \label{vec} h_i(0)=b_i(0)=0.\ee for $i=1,2,3,4$.  Moreover, the restriction of $\widetilde v_i$ on the central fiber is nothing but $v_i$. Note also that the independence of $h_i$ of $x$ and $y$ can also follow directly from the equivariance of $\pi$ since the differential of $\pi$ must map $\widetilde v_i$ to a vector field on $(\bC,0)$ and $\pi$ is just the projection on the third factor. On the other hand, $\widetilde v_i$'s must satisfy the same Lie bracket relations as in (\ref{liebracen}), i.e.
\be \label{liebratot}\left\{\begin{matrix}
[\widetilde v_1,\widetilde v_2]=0;[\widetilde v_1,\widetilde v_3]=\widetilde v_3; [\widetilde v_1,\widetilde v_4]=0\\
[\widetilde v_2,\widetilde v_3]=0;[\widetilde v_2,\widetilde v_4]=\widetilde v_4;[\widetilde v_3,\widetilde v_4]=0
\end{matrix}\right.\ee

\begin{lemma}\label{forfirord} With the above settings, one has \[
\left\{\begin{aligned}
\tilde v_1
&= \left[x+s\left(x^2f_1^1+xg_1^1\right)\right]\partial_x +s\partial_s, \\[4pt]
\tilde v_2
&= s\left(x^2f_1^2+xg_1^2\right)\partial_x +s\partial_s, \\[4pt]
\tilde v_3
&= \left[x^2 +s (x^2f_1^3+xg_1^3)\right]\,\partial_x
    \\[4pt]
\tilde v_4
&= s\left(x^2 f_1^4+xg_1^4-1\right)\,\partial_x.
\end{aligned}
\right.
\text{and  \;  }\left\{\begin{aligned}
\tilde v_1
&= -tyg_1^1\partial_y +t\partial_t, \\[4pt]
\tilde v_2
&= \left(y-tyg_1^2\right)\partial_y +t\partial_t, \\[4pt]
\tilde v_3
&= \left(-t- t yg_1^3\right)\,\partial_y
    \\[4pt]
\tilde v_4
&= \left(y^2 -tyg_1^4\right)\,\partial_y.
\end{aligned}
\right.
\] on $U_x$ and $U_y$ up to the first order, respectively.
\end{lemma}
\begin{proof}
On the $U_x$-chart, since $\left. \widetilde v_1\right|_{X_0\cap U_x}=v_1=x\partial_x$, one can write 
$$\left. \widetilde v_1\right|_{ U_x}=\left[x+s\left(x^2f_1^1+xg_1^1+h_1^1\right) \right] \partial_x+b_1^1s\partial_s$$ for the sake of Lemma \ref{vecgenfor}. Hence, 
$$\left. \widetilde v_1\right|_{ U_y}=\left[-y^2h_1^1-y(1-b_1^1) -tyg_1^1 \right] \partial_y+b_1^1t\partial_t$$ by the constraints (\ref{const}). However, $$\left. \widetilde v_1\right|_{X_0\cap U_y}=\left. v_1\right|_{U_y}=0.$$ Thus, $h_1^1=0$ and $b_1^1=1$ so that $$\tilde v_1
= -tyg_1^1\partial_y +t\partial_t $$ on $U_y$ so that
$$ \widetilde v_1=\left(x+sxg_1^1 \right) \partial_x+s\partial_s$$ on $U_x$ by the transform (\ref{vectrajac}).
The same strategy can be applied verbatim to the other vectors.
\end{proof}

Actually, more can be said about the extended $G$-action on the total space $(\mathcal{X},0)$ not just up to the first order.
\begin{prop} \label{veruni} The vector fields $v_3$ and $v_4$ corresponding to the unipotent part of $G$ are vertical, i.e. $b_1(s)=b_2(s)\equiv 0$ while the vector fields $v_1$ and $v_2$ corresponding to the reductive part of $G$ are not. In particular, the unipotent action of $\mathbb{C}\times \mathbb{C}$ if extended must act trivially on the base $(\mathbb{C},0)$ of the semi-universal family.
\end{prop}

\begin{proof}
Let $\mathfrak{h}$ be Lie algebra of vector fields on $(\bC,0)$. Consider the map
\begin{align*}
\rho: \mathfrak{g} &\rightarrow \mathfrak{h}\\
\widetilde v_i &\mapsto e_i:=b_i(s)\frac{\partial}{\partial_s}
\end{align*} which is actually a well-defined Lie algebra homomorphism since the first two components $\frac{\partial}{\partial_x}$ and $\frac{\partial}{\partial_y}$ does not contribute to the component $\frac{\partial}{\partial_s}$ in the Lie brackets. Therefore, the vector fields $e_i$'s satisfy the same Lie bracket relation as in (\ref{liebratot}). Observe that on $\mathfrak{h}$, one can introduce a filtration $F$ given by the vanishing order at $0$ and then two facts that $$[F^p\mathfrak{h},F^p\mathfrak{h}]\subset F^{2p}\mathfrak{h} \text{ and } [F^p\mathfrak{h},F^q\mathfrak{h}]\subset F^{p+q-1}\mathfrak{h} $$ can be easily shown. Now, let $b_i(s)=\sum_{j\geq 1}b_j^i s^j$. Using these facts and the Lie bracket relations $$[e_1,e_3]=e_3 \text{ and } [e_2,e_4]=e_4  ,$$ we get that $b_1^3=b_1^4=0$. 
Besides, it follows from Lemma \ref{forfirord} that $b_1^1=b_1^2=1$. In particular, it implies that the action of the reductive part $(\bC^* \times \bC^*)$ can not be trivial on the base. This turns out to be the main reason for the triviality of the unipotent action on the base. Indeed, we are going to show that $b_j^3=0$ for all $j\geq 1$ by induction. Note that $b_1^3=0$ already. First we compute the Lie bracket  \[
0=[e_2,e_3]= \sum_{j\ge 1} \sum_{l\ge 1} (l-j) \, b_j^2 \, b_l^3 \, s^{\,j+l-1} \, \partial_s.
\] We argue on the order of $s$. Looking at the coefficient of $s^2$, we have $$0=(1-2)b_2^2b_1^3+(2-1)b_1^2b_2^3=b_1^2b_2^3=b_2^3$$ which justifies the base case. Suppose that $b_j^3=0$ for $1 \leq j \leq k$. Looking at the coefficient of $s^{k+1}$, we have
$$0=\sum_{j+l-1=k+1} (l-j) \, b_j^2 \, a_l^3 =kb_1^2b_{k+1}^3=b_{k+1}^3$$ which ends the induction argument. Thus, $b_3$ is trivial. The triviality of $b^4$ can be shown in the same way using the Lie bracket relation $[e_1,e_4]=0$.
\end{proof}
\begin{claim} The local $G$-action on $(X,0)$ can not be extended to local $\bC$-actions on $(\mathcal{X},0)$ and on $(\mathbb{C},0)$, with respect to which the morphism $\pi$ is equivariant, where $G$ is the non-reductive group complex Lie group $(\bC^* \times \bC^*) \rtimes (\bC \times \bC)$.
\end{claim}
\begin{proof}
Suppose by contradiction that this is the case. Then by Lemma \ref{forfirord}, we can assume that  $$\tilde v_3 = \left[x^2 +sxg_1^3 \right]\partial_x \text{ and }\tilde v_4 = \left[s(xg_1^4-1)\right]\partial_x.$$ Their Lie bracket up to the first order is given by $$(-sg_1^4+2sx)\partial_x$$ which can never be $0$ no matter how one chooses $g_1^4$. Hence, the Lie bracket relation $[\tilde v_3,\tilde v_4]=0$ can not be preserved - a contradiction.
\end{proof}

Finally, we give examples of extendable non-reductive group actions. Let $G=(\bC^* \times \bC^*) \rtimes \bC$. Consider once again the germ of complex spaces $(X,0)$ given
Example \ref{ex: unicor} but this time equipped with the following local holomorphic $G$-action \begin{align*}
(G,\mathrm{id})\times (X,0) &\rightarrow (X,0)\\
((\lambda,\mu,a),(x,y)) &\mapsto \left(\frac{\lambda x}{1-a x},\mu y\right)
\end{align*} which generates four vector fields $v_1:=x\partial_x,v_2:=y\partial_y,v_3:=x^2\partial_x$ together with the following Lie bracket relations  \be \label{ccc}
[v_1,v_2]=0;[v_1,v_3]=v_3; [v_2,v_3]=0\ee
By using the vector fields $\tilde v_1, \tilde v_2,\tilde v_3$ given in Lemma \ref{forfirord}, subject to the same Lie bracket relations as in (\ref{ccc}) up to the first order, up to the first order, one can show that
\[
\left\{\begin{aligned}
\tilde v_1
&= x\partial_x +s\partial_s, \\[4pt]
\tilde v_2
&=s\partial_s, \\[4pt]
\tilde v_3
&= x^2 \,\partial_x.
\end{aligned}
\right.
\text{and  \;  }\left\{\begin{aligned}
\tilde v_1
&= t\partial_t, \\[4pt]
\tilde v_2
&= y\partial_y +t\partial_t, \\[4pt]
\tilde v_3
&= -t\,\partial_y
\end{aligned}
\right.
\] on  $U_x$ and $U_y$, respectively. Surprisingly, it can be easily checked that these vector fields are defined on $\mathcal{X}$, subject exactly to the Lie bracket relation (not just to the first order). Therefore, by integrating, they give rise to a holomorphic local $G$-action given by 
\begin{align*}
(G,\mathrm{id})\times (\mathcal{X},0) &\rightarrow (\mathcal{X},0)\\
((\lambda,\mu,a),(x,y,s)) &\mapsto \left(\frac{\lambda x}{1-a x},\mu y(1-ax),\lambda\mu s\right).
\end{align*} 
As expected from Proposition \ref{veruni}, the unipotent part acts trivial on the base while the reductive part does not. Finally, by symmetry, the same analysis can applied to the local holomorphic $G$-action on $(X,0)$ \begin{align*}
(G,\mathrm{id})\times (X,0) &\rightarrow (X,0)\\
((\lambda,\mu,a),(x,y)) &\mapsto \left(\lambda x,\frac{\mu y}{1-a y}\right)
\end{align*} where $G$ is still $(\bC^* \times \bC^*) \rtimes \bC$.
\appendix
\section{Analytic equations and approximation theorems}\label{sec: anaequapp}
\subsection{Artin's approximation theorems}

 Let $\bC\lbrace x, y \rbrace:=\bC\lbrace x_1,\ldots,x_n, y_1,\ldots,y_N\rbrace $  be the ring of convergent power series in $(x,y)$. Consider a system of analytic equations
\be \label{ae} f(x,y)=0
\ee where $f(x,y)=(f_1(x,y),\ldots ,f_m(x,y))$ is a vector with values in $\bC\lbrace x, y \rbrace $.
\begin{theorem}{{\rm{(}}\cite[Theorem 1.2]{Art68}{\rm{)}}}  If $\overline{y}(x)=(\overline{y}_1(x),\ldots,\overline{y}_N(x))$ with $\overline{y}_\nu(x) \in \bC[[x]]$ is a formal solution without constant term, solving (\ref{ae}), i.e., $f(x,\overline{y}(x))=0$. Then for any integer $c$, there exists a convergent solution $y(x=(y_1(x),\ldots,y_N(x)))$ of (\ref{ae}) such that $$y(x)\equiv \overline{y}(x) (\mod \mathfrak{m}^c)$$ where $ \mathfrak{m}^c$ is the maximal ideal of the ring $\mathbb{C}[[x]]$ of formal power series in $x$.
\end{theorem}
\begin{thrm}  {{\rm{(}}\cite[Theorem 1.3]{Art69}{\rm{)}}} \label{artappr} Let $\mathfrak{a}_1,\ldots ,\mathfrak{a}_N$ be proper ideals of $\bC\lbrace x\rbrace $. Suppose $\bar y(x)$ is a formal solution of (\ref{ae}) of the form 
$$\bar y_\nu(x)=u_\nu (x) +\bar v_\nu(x)$$ where $u_\nu (x) \in \bC\lbrace x \rbrace$ and $\bar v_\nu(x) \in \bC[[x]]$ such that $$\bar v_\nu(x) \equiv 0 (\mod \hat{\mathfrak{a} }_\nu:= \mathfrak{a}_\nu \bC[[x]]).$$  Then there exists a convergent solution $y(x)\in \bC\lbrace x\rbrace $ such that $y_\nu(x )\equiv u_\nu(x) (\mod \hat{\mathfrak{a} }_\nu).$
\end{thrm}
With the same setting, let $G$ be a reductive algebraic group (i.e. linear and every rational representation of $G$ is completely reducible). Suppose that $G$ acts linearly on $V=\bC^n$ and $W=\bC^N$. An element $\bar y(x)\in \bC[[x]]^N$ is said to be equivariant if $$\bar y(gx)=g\bar y(x), \;\forall g\in G.$$
\begin{theorem} {{\rm{(}}\cite[Theorem A]{BM79}{\rm{)}}}  \label{equiart}If $\bar y(x)$ is an equivariant formal power series solution of (\ref{ae}). Then for any $c\in \mathbb{N}$, there exists an equivariant convergent series solution $y(x)$ such that $y(x)\equiv \overline{y}(x) (\mod \mathfrak{m}^c)$.
\end{theorem}
\subsection{Grauert's Approximation Theorem} \label{secgr}
\begin{defn} {{\rm{(}}\cite{Gra72} \cite[Definition 8.2.1]{JP00}{\rm{)}}}  We consider four sets of variables, $x=(x_1,\cdots,x_n)$, $s=(s_1,\cdots,s_l)$, $\Phi=(\Phi_1,\cdots,\Phi_p)$,$\Psi=(\Psi_1,\cdots,\Psi_q)$. Let $I\subset \mathbb{C}\lbrace s\rbrace$ be an ideal. We can extend $I$ to an ideal $I \in \mathbb{C}\lbrace x,s\rbrace $. Consider an element $F=(F_1,\cdots,F_k)\in \mathbb{C}\lbrace x,s,\Phi,\Psi\rbrace^k$
\begin{enumerate}
\item[$(1)$] An analytic solution of the equation $F\equiv 0$ modulo $I$ is a pair $(\phi,\psi)\in \mathbb{C}\lbrace s\rbrace^p \times \mathbb{C}\lbrace x,s\rbrace^q $, that is power series $\phi_1,\cdots,\phi_p\in \mathbb{C}\lbrace s\rbrace$ and $\psi_1,\cdots,\psi_q\in \mathbb{C}\lbrace x ,s\rbrace$ such that 
$$F(x,s,\phi(s),\psi(x,s))\equiv 0\mod I.\mathbb{C}\lbrace x ,s\rbrace^k,$$
i..e, $$F_1(x,s,\phi(s),\psi(x,s)),\cdots,F_k(x,s,\phi(s),\psi(x,s)) \in I.$$
\item[$(2)$] A solution of order $e$ of the equation $F\equiv 0 \mod I$ is a pair $(\phi,\psi)\in \mathbb{C}[ s]^p \times \mathbb{C}\lbrace x \rbrace[s]^q $  such that 
$$F(x,s,\phi(s),\psi(x,s))\equiv 0\mod (I+\mathfrak{m}^{e+1}).\mathbb{C}\lbrace x ,s\rbrace^k$$ where $\mathfrak{m}=(s_1,\cdots,s_l)$.
\end{enumerate}
\end{defn}
\begin{thrm}[\textbf{Grauert's Approximation Theorem}]{{\rm{(}}\cite{Gra72}, \cite[Theorem 8.2.2]{JP00}{\rm{)}}}\label{grauapp} Let the notations be as in the previous definition. Let $e_0\in \mathbb{N}.$ Suppose that the system of equations $$F\equiv 0 \mod I$$ has a solution  $(\phi^{(e_0)},\psi^{(e_0)})$ of order $e_0$. Suppose further that for all $e\geq e_0$ every solution $(\phi^{(e)},\psi^{(e)})$ order $e$ with $\phi^{(e)}=\phi^{(e_0)}\mod \mathfrak{m}^{e+1} $ and $\psi^{(e)}=\psi^{(e_0)}\mod \mathfrak{m}^{e+1} $ extends to a solution  $(\phi^{(e)}+\delta^{(e)},\psi^{(e)}+\gamma^{(e)})$ of oder $e+1$ with $\delta^{(e)}\in \mathbb{C}[ s]^p$ and $\gamma^{(e)}\in \mathbb{C}\lbrace x \rbrace [s]^p$ homogeneous of degree $e+1$ in $s$. Then the system of equation 
$$F \equiv 0 \mod I$$ has an analytic solution  $(\phi,\psi)$ with $\phi=\phi^{(e_0)}\mod \mathfrak{m}^{e+1} $ and $\psi=\psi^{(e_0)}\mod \mathfrak{m}^{e+1} $.
\end{thrm}

\end{document}